\documentclass[12pt]{amsart}
\usepackage{latexsym,amsmath,amsfonts,amscd,amssymb}
\usepackage{stmaryrd}
\usepackage{array}   
\usepackage[mathscr]{eucal} 
\usepackage{mathrsfs}
\usepackage{hyperref}
\usepackage{graphicx}
\usepackage{xypic}
\usepackage{calligra}
\usepackage[utf8]{inputenc}
\usepackage[T1]{fontenc}
\DeclareFontFamily{OT1}{pzc}{}
\DeclareFontShape{OT1}{pzc}{m}{it}{<-> s * [1.10] pzcmi7t}{}
\DeclareMathAlphabet{\mathpzc}{OT1}{pzc}{m}{it}
\DeclareFontFamily{OT1}{rsfs}{}
\DeclareFontShape{OT1}{rsfs}{n}{it}{<->rsfs10}{}
\DeclareMathAlphabet{\curly}{OT1}{rsfs}{n}{it}
\theoremstyle{plain}
\newtheorem{theorem}{Theorem}[section]
\newtheorem*{theorem*}{Theorem}

\newtheorem{corollary}[theorem]{Corollary}
\newtheorem{lemma}[theorem]{Lemma}
\newtheorem{proposition}[theorem]{Proposition}

\theoremstyle{definition}
\newtheorem{definition}[theorem]{Definition}
\theoremstyle{remark}
\newtheorem{example}[theorem]{Example}
\newtheorem{remark}[theorem]{Remark}

\newtheorem*{claim*}{Claim}
\numberwithin{equation}{section}

\renewcommand{\le}{\leqslant}

\renewcommand{\ge}{\geqslant}

\newcommand{\R}{\mathbb{R}}

\newcommand{\Z}{\mathbb{Z}}
\newcommand{\C}{\mathbb{C}}

\DeclareMathOperator*{\Ima}{Im}

\newcommand{\gen}[1]{\left< #1 \right>}

\newcommand{\lie}{\mathfrak}

\newcommand{\U}{\mathrm{U}}

\DeclareMathOperator{\ad}{ad}
\DeclareMathOperator{\Ad}{Ad}

\DeclareMathOperator{\tr}{tr}

\DeclareMathOperator{\rank}{rank}

\DeclareMathOperator{\Hom}{Hom}
\DeclareMathOperator{\End}{End}

\DeclareMathOperator{\Id}{Id}

\DeclareMathOperator{\Aut}{Aut}
\DeclareMathOperator{\Int}{Int}

\DeclareMathOperator*{\aut}{Aut}

\DeclareMathOperator*{\GL}{GL}
\DeclareMathOperator*{\SL}{SL}
\DeclareMathOperator*{\SO}{SO}
\DeclareMathOperator*{\SU}{SU}
\DeclareMathOperator*{\Sp}{Sp}

\renewcommand{\phi}{\varphi}

\newcommand{\liec}{\mathfrak{c}}
\newcommand{\liem}{\mathfrak{m}}
\newcommand{\liep}{\mathfrak{p}}
\newcommand{\liet}{\mathfrak{t}}

\newcommand{\lieh}{\mathfrak{h}}

\newcommand{\lieg}{\mathfrak{g}}
\newcommand{\liee}{\mathfrak{e}}
\newcommand{\liek}{\mathfrak{k}}

\newcommand{\liez}{\mathfrak{z}}
\newcommand{\liegl}{\mathfrak{gl}}
\newcommand{\liesu}{\mathfrak{su}}
\newcommand{\lieso}{\mathfrak{so}}
\newcommand{\liesp}{\mathfrak{sp}}
\newcommand{\liesl}{\mathfrak{sl}}

\renewcommand{\phi}{\varphi}

\usepackage{tikz-cd}
\usetikzlibrary{calc}
\tikzset{curve/.style={settings={#1},to path={(\tikztostart)
			.. controls ($(\tikztostart)!\pv{pos}!(\tikztotarget)!\pv{height}!270:(\tikztotarget)$)
			and ($(\tikztostart)!1-\pv{pos}!(\tikztotarget)!\pv{height}!270:(\tikztotarget)$)
			.. (\tikztotarget)\tikztonodes}},
	settings/.code={\tikzset{quiver/.cd,#1}
		\def\pv##1{\pgfkeysvalueof{/tikz/quiver/##1}}},
	quiver/.cd,pos/.initial=0.35,height/.initial=0}
\tikzset{tail reversed/.code={\pgfsetarrowsstart{tikzcd to}}}
\tikzset{2tail/.code={\pgfsetarrowsstart{Implies[reversed]}}}
\tikzset{2tail reversed/.code={\pgfsetarrowsstart{Implies}}}
\tikzset{no body/.style={/tikz/dash pattern=on 0 off 1mm}}

\begin{document}
\title[Cyclic Higgs bundles and the Toledo invariant]
{Cyclic Higgs bundles and the Toledo invariant}

\author[Oscar Garc{\'\i}a-Prada]{Oscar Garc{\'\i}a-Prada}
\address{Instituto de Ciencias Matem\'aticas \\
  CSIC-UAM-UC3M-UCM \\ Nicol\'as Cabrera, 13--15 \\ 28049 Madrid \\ Spain}
\email{oscar.garcia-prada@icmat.es}
\author[Miguel González]{Miguel González}
\address{Instituto de Ciencias Matem\'aticas \\
	CSIC-UAM-UC3M-UCM \\ Nicol\'as Cabrera, 13--15 \\ 28049 Madrid \\ Spain}
\email{miguel.gonzalez@icmat.es}
\noindent
\thanks{
 \noindent
Partially supported by the Spanish Ministry of Science and
Innovation, through the ``Severo Ochoa Programme for Centres of Excellence in R\&D (CEX2019-000904-S)'' and grants  PID2019-109339GB-C31 and PID2022-141387NB-C21. The second author was supported by JAE Intro ICU grant JAEICU-21-ICMAT-01 for master's studies financed through ICMAT Severo Ochoa grant No. CEX2019-000904-S and by a fellowship from "la Caixa" Foundation (ID 100010434) with code LCF/BQ/DR23/12000030. This version of the article has been accepted for publication after peer review but is not the Version of Record and does not reflect post-acceptance improvements, or any corrections. The Version of Record is available online at: \href{https://dx.doi.org/10.1007/s00031-026-09962-2}{https://dx.doi.org/10.1007/s00031-026-09962-2}}


\begin{abstract}
  Let $G$ be a complex semisimple Lie group and $\lieg$ its Lie algebra. In this paper, we study a special class of cyclic Higgs bundles constructed from a $\Z$-grading $\lieg = \bigoplus_{j=1-m}^{m-1}\lieg_j$ by using the natural representation $G_0 \to \GL(\lieg_1 \oplus \lieg_{1-m})$, where $G_0 \le G$ is the connected subgroup corresponding to $\lieg_0$. The resulting Higgs pairs include $G^\R$-Higgs bundles for $G^\R \le G$ a real form of Hermitian type (in the case $m=2$) and fixed points of the $\C^*$-action on $G$-Higgs bundles (in the case where the Higgs field vanishes along $\lieg_{1-m}$). In both of these situations a topological invariant with interesting properties, known as the Toledo invariant, has been defined and studied in the literature. This paper generalises its definition and properties to the case of arbitrary $(G_0,\lieg_1 \oplus \lieg_{1-m})$-Higgs pairs, which give rise to families of cyclic Higgs bundles. The results are applied to the example with $m=3$ that arises from the theory of quaternion-Kähler symmetric spaces.
\end{abstract}

\maketitle

\section{Introduction}

Let $G$ be a complex semisimple Lie group with Lie algebra $\lieg$. A $G$-Higgs bundle is a pair $(E,\varphi)$ where $E$ is a holomorphic principal $G$-bundle over a compact Riemann surface $X$ of genus $g \ge 2$, and $\varphi \in H^0(X, E(\lieg) \otimes K_X)$, where $K_X$ is the canonical line bundle of $X$ and $E(\lieg) := E \times_{\Ad} \lieg$ is the adjoint bundle. Suitable notions of stability give rise to the moduli space of $G$-Higgs bundles, $\mathcal M(G)$, which has been extensively studied and shown to exhibit a very rich geometry and many interesting properties. One of the major aspects is the existence of the \textit{nonabelian Hodge correspondence} \cite{corlette_flat_1988, donaldson_twisted_1987, hitchin_self-duality_1987, simpson_hodge_1996, simpson_higgs_1992}, via which $\mathcal M(G)$ is homeomorphic to the variety of $G$-characters of $\pi_1(X)$, whose elements are the (completely reducible) representations $\rho : \pi_1(X) \to G$.

This paper focuses on the study of the subvarieties of \textit{cyclic Higgs bundles} in $\mathcal M(G)$. These subvarieties are obtained by choosing a finite order automorphism $\theta \in \Aut_m(G)$, as well as a primitive $m$-th root of unity $\zeta_m \in \C^*$, and considering the fixed points of the action of the finite cyclic group $\Z/m\Z$ on $\mathcal M(G)$ generated by
\[(E,\varphi) \mapsto (\theta(E), \zeta_m \cdot \theta(\varphi)),\]
\noindent where $\theta$ also denotes the induced finite order element in $\Aut(\lieg)$.

Cyclic Higgs bundles were originally considered by Simpson \cite{simpson_katz_2006}, under the name of \textit{cyclotomic Higgs bundles}, in relation to the study of local systems. These were defined as fixed points for the action of $\Z/m\Z$ given by scaling the Higgs field by $\zeta_m$, which is included in our case of study by considering $\theta$ to be the identity (or any inner automorphism, by gauge equivalence). They have also been considered in \cite{baraglia_cyclic_2015}, where they are used to provide solutions for the field theoretic \textit{affine Toda equations}. Other appearances of cyclic Higgs bundles in the literature include \cite{collier_finite_2016, collier_holomorphic_2023, li_cyclic_2017, garcia-prada_vinberg_unpublished, labourie_cyclic_2017, li_noncompact_2020}.

Our approach to the study of cyclic Higgs bundles relies on their description given in \cite{garcia-prada_involutions_2019} (see also \cite{garcia-prada_vinberg_unpublished}) by means of the theory of representations associated with $\Z/m\Z$-gradings of Lie algebras. The automorphism $\theta$ induces such a grading
\[\lieg = \bigoplus_{j \in \Z/m\Z}\bar{\lieg}_j.\]
The fixed point subgroup $G^\theta \le G$ has $\bar{\lieg}_0$ as Lie algebra, and the adjoint action of $G$ restricts to a representation $G^\theta \to \GL(\bar{\lieg}_j)$ on each of the pieces of the grading. These representations $(G^\theta, \bar{\lieg}_j)$ are called \textit{Vinberg $\theta$-pairs} and were studied in depth in \cite{vinberg_graded}.

Cyclic Higgs bundles are then obtained from the image of the map $\mathcal M(G^\theta, \bar{\lieg}_1) \to \mathcal M(G)$, where $\mathcal M(G^\theta, \bar{\lieg}_1)$ is the moduli space of polystable $(G^\theta, \bar{\lieg}_1)$-Higgs pairs (i.e. pairs $(E,\varphi)$ comprised of a holomorphic $G^\theta$-bundle $E$ over $X$ and a section $\varphi \in H^0(X, E(\bar{\lieg}_1) \otimes K_X)$) and the map is given by extension of the structure group to $G$. In \cite{garcia-prada_involutions_2019} it is also shown that every stable and simple cyclic Higgs bundle is obtained in this way, by considering the images of these maps for the different automorphisms $\theta'$ in the same inner class as $\theta$, up to conjugation. 

The goal of the present work is to introduce a topological invariant for the moduli spaces $\mathcal M(G_0,\bar{\lieg}_1)$, where $G_0 \subseteq G^\theta$ is the connected component of the identity, and study its properties. For this we follow the approach of \cite{biquard_arakelov-milnor_2021}, where the theory of $\Z$-gradings of the Lie algebra $\lieg$ allows to introduce such an invariant in a different situation, which we now recall.

Given a $\Z$-grading $\lieg = \bigoplus_{i \in \mathbb Z}\lieg_i$, and denoting by $G_0 \le G$ the connected subgroup corresponding to the Lie subalgebra $\lieg_0$, we have as before that the adjoint representation restricts, giving $\rho : G_0 \to \GL(\lieg_i)$. For $i\neq 0$, the action has been studied by Vinberg (see e.g. \cite[Chapter X]{knapp}) and shown to possess a unique open orbit, which is dense. 

Via the previous representations, which we call \textit{Vinberg $\C^*$-pairs}, it is possible to consider $(G_0,\lieg_i)$-Higgs pairs $(E,\varphi)$, where $E$ is now a principal $G_0$-bundle over $X$ and $\varphi \in H^0(X,E(\lieg_i) \otimes K_X)$. Again, there is a moduli space of polystable pairs $\mathcal M(G_0, \lieg_i)$. These pairs provide \textit{Hodge bundles}, that is, fixed points of the natural $\C^*$-action on $\mathcal M(G)$ given by scaling the Higgs field \cite[Section 4.1]{biquard_arakelov-milnor_2021}.

In \cite{biquard_arakelov-milnor_2021}, it is shown that $(G_0, \lieg_1)$-Higgs pairs possess a topological invariant with interesting properties, known as the \textit{Toledo invariant}. It is defined by considering the \textit{Toledo character} $\chi_T : \lieg_0 \to \C$ given by $x \mapsto B^*(\gamma,\gamma)B(\zeta,x)$, where $B$ is an $\Ad$-invariant form on $\lieg$ (such as the Killing form), $\zeta \in \lieg_0$ is an element such that $\lieg_i = \ker(\ad(\zeta) - i\Id_{\lieg}) \subseteq \lieg$, and $B^*(\gamma,\gamma)$ is a normalisation term. The Toledo invariant $\tau(E,\varphi)$ is then the degree of $E$ with respect to $\chi_T$, appropriately defined. The work of \cite{biquard_arakelov-milnor_2021} shows a Milnor--Wood-type bound for $\tau$ in $\mathcal M(G_0,\lieg_i)$, known as \textit{Arakelov--Milnor inequality}, and exhibits a description of the locus attaining the bound.

In order to relate the theory of $\Z$-gradings to cyclic Higgs bundles, we will assume that there is a $\Z$-grading of the Lie algebra, of the form
\[\lieg = \lieg_{1-m} \oplus \dots \oplus \lieg_{m-1},\]
\noindent such that $\bar{\lieg}_0 = \lieg_0$ and $\bar{\lieg}_j = \lieg_j \oplus \lieg_{j-m}$ for $j \in \{1,\dots,m-1\}$. In this case, the group $G_0$ given by the $\Z$-grading is the same as the connected component of the identity $G_0 \subseteq G^\theta$, whence the identical notation. This assumption forces the automorphism $\theta \in \Aut_m(\lieg)$ to be inner, and we characterise which inner automorphisms fall under these conditions. For example, in type $A_n$, this property is satisfied by every inner automorphism. 

In this situation, the $(G_0,\bar{\lieg}_1)$-Higgs pairs $(E,\varphi)$ that we are interested in are related to the previously discussed $(G_0,\lieg_1)$-Higgs pairs $(E,\varphi^+)$, but now we allow the addition of a term $\varphi^- \in H^0(X,E(\lieg_{1-m}) \otimes K_X)$ to the Higgs field. The space $\lieg_{1-m}$ is special because, if the grading is induced by an $\liesl_2\mathbb C$-triple $(h,e,f)$ of $\lieg$ (i.e. if $\zeta = \ad(\frac{h}{2})$), then $\lieg_{1-m}$ is the sum of the lowest weight spaces for the given $\liesl_2\mathbb C$-representation. We prove that, due to this, the results of \cite{biquard_arakelov-milnor_2021} extend to the moduli space $\mathcal M(G_0,\bar{\lieg}_1)$, that is, there is a Toledo invariant satisfying the Arakelov--Milnor inequality, and the space of $(G_0, \bar{\lieg}_1)$-Higgs pairs attaining the bound can be described via a \textit{Cayley correspondence}.

This situation has already been extensively studied for $m=2$. Gradings $\lieg=\lieg_1 \oplus \lieg_{0} \oplus \lieg_{-1}$ arise from the complexification of Cartan decompositions $\lieg^\R = \lieh^\R \oplus \liem^\R$, where $\lieg^\R$ is the Lie algebra of a real form $G^\R \le G$ of Hermitian type (i.e. the associated symmetric space $G^\R/H^\R$ is Hermitian), $\lieh = \lieg_0$ and $\liem = \lieg_1 \oplus \lieg_{-1}$ (where $\lieg = \lieh \oplus \liem$ is the complexified decomposition). The inequality for the Toledo invariant and the Cayley correspondence in this case were established in \cite{bradlow_classicalhermitian_2006, bradlow_sostar_2015, bradlow_surface_2003, garcia-prada_symplectic_2013} (for classical real forms) and \cite{biquard_higgs_2017} (in general).

The case of $m=3$ is also connected to symmetric spaces, as it relates to those of quaternionic type \cite{bertram_quaternionic_2002,wallach_quaternionic_1996}, which exist for every complex simple Lie algebra $\lieg$. Indeed, any such algebra possesses a $\Z$-grading $\lieg = \lieg_{-2} \oplus \lieg_{-1} \oplus \lieg_0 \oplus \lieg_1 \oplus \lieg_2$ which is given by the eigenspaces of $\ad(T_\beta)$, where $T_\beta$ is a coroot vector for the highest root $\beta$. The resulting symmetric space for $(\lieg_{-2} \oplus \lieg_0 \oplus \lieg_2) \oplus (\lieg_{-1} \oplus \lieg_1)$ has the special property of being quaternion-Kähler, that is, having holonomy in $\Sp(n) \cdot \Sp(1) \subseteq \SO(4n)$. Conversely, this classifies all quaternion-Kähler symmetric spaces. Our results provide bounds for the Toledo invariant and a Cayley correspondence for the moduli space $\mathcal M(G_0, \lieg_1 \oplus \lieg_{-2})$.

This paper is structured as follows. In Section \ref{zgrad} we introduce the required theory of gradings of Lie algebras and Vinberg pairs. We also define the Toledo character and collect some of its properties. In Section \ref{higgstoledo} we define Higgs bundles and their notions of stability, we see how our setting relates to cyclic Higgs bundles, we detail the main ingredients of non-abelian Hodge theory in our context and we define the Toledo invariant for $\mathcal M(G_0,\lieg_1 \oplus \lieg_{1-m})$, proving its bound in Theorem \ref{amw}. In Section \ref{cayley} we study the locus of elements in $\mathcal M(G_0, \lieg_1 \oplus \lieg_{1-m})$ attaining the bound in the \textit{JM-regular} case (a technical requirement generalising the \textit{tube type} condition for Hermitian symmetric spaces). We relate this locus, via the Cayley correspondence of Theorem \ref{caycorr} and Proposition \ref{cayleysurj} to the moduli space of polystable $K_X^m$-twisted $(C,V)$-Higgs pairs, where $C \le G_0$ is a reductive subgroup and $V \subseteq \lieg_0$ is a vector subspace, naturally isomorphic to $\lieg_{1-m}$, on which $C$ acts. Finally, in Section \ref{quater}, we apply the results to the quaternionic $\Z$-grading of a complex simple Lie algebra $\lieg$ and show that the inequalities for the Toledo invariant and the Cayley correspondence can be improved in this situation.

\textbf{Acknowledgements.} We would like to thank Brian Collier and Alastair King for very useful discussions on aspects of this project. 

\section{Gradings of Lie algebras and the Toledo character}\label{zgrad}

For this section we let $\lieg$ be a finite-dimensional semisimple complex Lie algebra and $G$ a Lie group with Lie algebra $\lieg$.

\subsection{$\Z$-gradings of semisimple Lie algebras}\label{zgradings}
\begin{definition}
	A \textbf{$\Z$-grading} of $\lieg$ is a decomposition as a direct sum of vector subspaces
	\[\lieg = \bigoplus_{j \in \Z}\lieg_j\]
	\noindent such that $[\lieg_j, \lieg_k] \subseteq \lieg_{j+k}$.  
\end{definition}

The linear map $D: \lieg \to \lieg$ given by $D|_{\lieg_k} = k\Id_{\lieg_k}$ is a Lie algebra derivation. Semisimplicity of $\lieg$ implies that $D = \ad \zeta$ for some $\zeta \in \lieg$ which must then belong to $\lieg_0$. Such an element is called \textbf{grading element}.

As $\lieg$ is finite-dimensional, all but finitely many terms in the decomposition must be zero. Given a $\Z$-grading of $\lieg$, we set $m \in \Z_{>0}$ to be the smallest positive integer such that $|j| \ge m$ implies $\lieg_j = 0$. In other words, we consider $\Z$-gradings of the form
\[\lieg = \lieg_{1-m} \oplus \dots \oplus \lieg_{m-1}.\]  
\noindent such that $\lieg_{m-1} \neq 0$.

\begin{example}[Symmetric pairs of Hermitian type]\label{realhermitian}
	The case $m=2$ is related to the theory of symmetric spaces of Hermitian type. Indeed, if $G^\R/H^\R$ is such a symmetric space, with $G^\R \le G$ a real form with Lie algebra $\lieg^\R$ and $H^\R \le G^\R$ its maximal compact subgroup with Lie algebra $\lieh^\R$, we have the Cartan decomposition $\lieg^\R = \lieh^\R \oplus \liem^\R$. This gives a complexified decomposition $\lieg = \lieh \oplus \liem$, and the Hermitian structure on $G^\R/H^\R$ (whose tangent space at the neutral element is identified with $\liem^\R$) decomposes $\liem = \liem^+ \oplus \liem^-$ in $\pm i$-eigenspaces. We then have $\lieg_0 = \lieh$, $\lieg_1 = \liem^+$, $\lieg_{-1} = \liem^-$. For $\lieg$ simple, the possibilities for the real form $\lieg^\R$ are $\liesu(p,q)$, $\liesp(2n,\R)$, $\lieso^*(2n)$, $\lieso(2,n)$, $\liee_6(-14)$ and $\liee_7(-25)$. See \cite[Section 2.1]{biquard_higgs_2017} for further details.
\end{example}

\begin{example}[Symmetric pairs of quaternionic type]\label{quaternionic}
	The case $m=3$ relates to symmetric spaces of quaternionic type, which are symmetric spaces $M = G^\R/H^\R$ having a distribution of \textit{twisted} quaternion algebras (see \cite[Section 14D]{besse2007einstein}). For each complex simple Lie group there is one such space, and this results in the fact that each simple complex Lie algebra $\lieg$ has a distinguished \textit{quaternionic} $\Z$-grading of the form $\lieg_{-2} \oplus \lieg_{-1} \oplus \lieg_0 \oplus \lieg_1 \oplus \lieg_2$. This example will be studied in more detail in Section \ref{quater}.
\end{example}

The relation $[\lieg_0,\lieg_0] \subseteq \lieg_0$ implies that $\lieg_0$ is a Lie subalgebra, so that there is a corresponding connected subgroup $G_0 \le G$, which is reductive. Moreover, $[\lieg_0, \lieg_j] \subseteq \lieg_j$ implies that the adjoint action of $G$ in $\lieg$ restricts to a representation $G_0 \to \GL(\lieg_j)$ for each $j$. We recall the following definition:

\begin{definition}
	Let $\rho : G \to \GL(V)$ be a holomorphic representation with an open $G$-orbit $\Omega \subseteq V$. The vector space $V$ is called a \textbf{prehomogeneous vector space} for $G$, and the pair $(G,V)$ is also called a \textit{prehomogeneous vector space}.
\end{definition}

In this case the open orbit is dense and unique \cite[Proposition 10.1]{knapp}. The simplest example is given by the standard action of $\GL_n(\C)$ on $\C^n$, where $\Omega = \C^n \backslash \{0\}$.

\begin{definition}
	Let $(H,W)$ and $(G,V)$ be two prehomogeneous vector spaces. If $H \subseteq G$ is a subgroup, $W \subseteq V$ is a vector subspace, and the action of $H$ in $W$ is obtained by restricting the action of $G$ on $V$, we say that $(H,W)$ is a \textbf{prehomogeneous vector subspace} of $(G,V)$.
\end{definition} 

Regarding the representations coming from $\Z$-gradings, we have the following result of Vinberg (see \cite[Theorem 10.19 and Corollary 10.22]{knapp}):

\begin{proposition}\label{finorbits}
	For $j \neq 0$, the action of $G_0$ on $\lieg_j$ has finitely many orbits. In particular, $(G_0, \lieg_j)$ is a prehomogeneous vector space.
\end{proposition}

This motivates the following definition.

\begin{definition}
	Pairs $(G_0,\lieg_j)$ for $j \neq 0$ are called \textbf{Vinberg $\C^*$-pairs}.
\end{definition}

\begin{remark}\label{regrading}
	We may assume that $j=1$, since otherwise we can consider the $\Z$-graded subalgebra $\lieg' := \bigoplus_{k\in \mathbb Z}\lieg_{kj}$ whose zero-th piece is still $\lieg_0$ and whose first piece is $\lieg_j$.
\end{remark}

\begin{example}[Representations of type $A_m$ quivers]\label{linear_quiver}
	Let $V$ be an $n$-dimensional complex vector space and consider $G = \SL_n(V)$. We fix $m \in \Z_{>0}$ and a direct sum decomposition $V = V_0 \oplus \dots \oplus V_{m-1}$, where $\dim V_i = d_i$ (so that $\sum_i d_i = n$). Then, the Lie algebra $\lieg = \lie{sl}(V)$ consists of traceless endomorphisms of $V$, which we grade by
	$$\lieg_k := \left(\bigoplus_{j = 0}^{m-1}\Hom(V_j, V_{j+k})\right)_0,$$
	
	\noindent where $V_j=0$ for $j \notin \{0,\dots,m-1\}$ and the subscript zero means taking the subspace of traceless endomorphisms, and is only meaningful for $\lieg_0$. This is alternatively defined by the grading element $\zeta \in \lieg_0 = (\bigoplus_{j = 0}^{m-1}\End(V_j))_0$ given by $\zeta|_{V_j} = (j-\alpha)\Id|_{V_j}$, where $\alpha$ is a fixed constant (depending on the $d_i$) so that the obtained map is traceless. Note that $m$ is chosen so that $\lieg_j = 0$ for $|j| \ge m$, as explained before.
	
	The associated Vinberg $\C^*$-pair is $(G_0,\lieg_1)$, where
	\[G_0 = S(\GL\nolimits_{d_0}(\mathbb C) \times \dots \times \GL\nolimits_{d_{m-1}}(\mathbb C)),\]
	\[\lieg_1 = \bigoplus_{j = 0}^{m-1}\Hom(V_j, V_{j+1}).\]
	
	The space $\lieg_1$ consists of representations of a quiver with $m$ vertices and arrows $i \mapsto i+1$ for $i \in \{0,\dots,m-2\}$ (a \emph{linear quiver}, or \emph{type $A_{m}$ quiver}) where we put $V_i$ on the $i$-th vertex: 
	\[\begin{tikzcd}
		{V_0} && {V_1} && \dots && {V_{m-1}}
		\arrow["{f_0}", from=1-1, to=1-3]
		\arrow["{f_1}", from=1-3, to=1-5]
		\arrow["{f_{m-2}}", from=1-5, to=1-7]
	\end{tikzcd}.\]	
	
	The orbits of the action of $G_0$ on such representations have been studied and are explained, for example, in \cite[Section 2]{abeasis_geometry_1981} (more precisely, they study the action of $\prod \GL_{d_i}(\mathbb C)$, but it can be checked that it has the same orbits as that of $G_0$). For an element $f \in \lie g_1$ as in the previous diagram, and $0 \le i < j \le m-1$, we denote $r_{ij} := \rank (f_{j-1} \circ \dots \circ f_{i})$ the rank of the consecutive composition, a linear map from $V_i$ to $V_j$. Then, each feasible choice of ranks $(r_{ij})_{0 \le i < j \le m-1}$, with $0 \le r_{ij} \le \min\{d_i,\dots,d_j\}$, determines a unique orbit. The open orbit is the one where ranks are maximal, that is, $r_{ij} = \min\{d_i,\dots,d_j\}$ for all $i,j$.
\end{example}

We now define a special class of Vinberg $\C^*$-pairs, called Jacobson--Morozov-regular, following \cite[Section 2.4]{biquard_arakelov-milnor_2021}.

Let $(G_0, \lieg_1)$ be a Vinberg $\C^*$-pair and $e \in \Omega \subseteq \lieg_1$ an element of the open orbit. Since $\ad(e)$ maps $\lieg_j$ to $\lieg_{j+1}$, we have $\ad^{2m-1}(e) \equiv 0$. By the Jacobson--Morozov theorem, we can complete $e$ to an $\liesl_2$-triple $(h,e,f)$, where we can take $h \in \lieg_0$ and $f \in \lieg_{-1}$ by \cite[Lemma 10.18]{knapp}.

\begin{definition}
	A Vinberg $\C^*$-pair $(G_0,\lieg_1)$ is \textbf{Jacobson--Morozov-regular} or \textbf{JM-regular} if there exists an $\lie{sl}_2$-triple $(h,e,f)$ with $e \in \Omega$ such that $\zeta := \frac{h}{2}$ is a grading element for the $\Z$-grading of $\lieg$. 
\end{definition}

The term \textit{regular} in the name refers to the fact that, in this case, the prehomogeneous vector space $(G_0,\lieg_1)$ is regular, in the sense that the stabiliser $G^e$ is reductive \cite[Corollary 2.7]{biquard_arakelov-milnor_2021}. Moreover, it is a consequence of Malcev--Kostant theorem \cite[Theorem 10.10]{knapp} that every element $e \in \Omega$ can be included into such a triple $(2\zeta,e,f)$. Also, if $e \in \Omega$ is chosen and $(h,e,f)$ is any $\lie{sl}_2$-triple with $h \in \lie g_0$ (not necessarily the one with $h=2\zeta$), then $h$ is conjugate to $2\zeta$ under the adjoint action of $G_0$, by \cite[Proposition 2.19]{biquard_arakelov-milnor_2021}.

Even if $(G_0,\lieg_1)$ is not JM-regular, we may construct an associated JM-regular prehomogeneous vector subspace. We have that $\ad(h)$ diagonalizes with integer eigenvalues, so there is another $\Z$-grading $\lieg = \bigoplus_{k \in \mathbb Z} \tilde{\lieg}_k$ given by the eigenspaces. Define
$$\hat{\lieg}_k := \lieg_k \cap \tilde{\lieg}_{2k},$$
\noindent yielding the subalgebra $\hat{\lieg} = \bigoplus_{k \in \Z} \hat{\lieg}_k$. Denote by $(\hat{G}_0, \hat{\lieg}_1)$ the corresponding Vinberg $\C^*$-pair.

Note that since $[h,\zeta] = 0$ we have $h,\zeta \in \hat{\mathfrak g}_0$, and since $\frac{1}{2}[h,e] = [\zeta, e]= e$ we also have $e \in \hat{\mathfrak{g}}_1$. Also, $\hat{\mathfrak g}$ is precisely the subalgebra of elements where $\zeta$ and $\frac{h}{2}$ coincide, that is, the stabiliser of $s := \zeta - \frac{h}{2}$ in $\hat{\mathfrak g}$. If $\hat{G}_0 \subseteq G_0$ is the centraliser of $h$ in $G_0$ (equivalently, the centraliser of $s$), then $(\hat{G}_0,\hat{\mathfrak g}_1)$ is a prehomogeneous vector subspace of $(G_0,\mathfrak g_1)$ which is JM-regular and $e \in \hat{\Omega} \subseteq \hat{\mathfrak g}_1$, the latter by Malcev--Kostant theorem \cite[Theorem 10.10]{knapp}.

\begin{definition}\label{maxphv}
	Given $e \in \lie g_1$, the prehomogeneous vector subspace $(\hat{G}_0, \hat{\lie g}_1)$ of $(G_0,\mathfrak g_1)$ constructed above is called a \textbf{maximal JM-regular prehomogeneous vector subspace} for $e$.
\end{definition}

\begin{remark}
	The term \textit{maximal} refers to the fact that, among the JM-regular prehomogeneous vector subspaces of $(G_0,\lieg_1)$ which can be constructed from a graded Lie subalgebra $\hat{\lieg} = \bigoplus_{i \in \Z} \hat{\lieg}_i \subseteq \lieg$ (preserving degree, i.e. $\hat{\lieg}_i \subseteq \lieg_i$) and such that $e$ belongs to the open orbit of $\hat{G}_0$ in $\hat{\lieg}_1$, the one constructed above is maximal. Indeed, the JM-regularity condition implies that $\frac{h}{2}$ must be a grading element for $\hat{\lieg} = \bigoplus_{i \in Z} \hat{\lieg}_i$, and the fact that the grading for $\hat{\lieg}$ is contained in that of $\lieg$ implies that $\ad(\zeta)|_{\hat{\lieg}} = \ad(\frac{h}{2})|_{\hat{\lieg}}$. Thus, $\hat{\lieg}$ must be centralised by $s = \zeta - \frac{h}{2}$. This shows that the maximal choice corresponds to the centraliser itself, leading to the previous construction. 
\end{remark}

\begin{example}
	In the $\Z$-grading from Example \ref{linear_quiver} corresponding to linear quiver representations, we can compute from a given partition of $n$ (giving the dimensions of each piece) whether the resulting pair is JM-regular. An element of the open orbit $\Omega \in \mathfrak g_1$ is given by $e = (e_0,\dots, e_{m-1})$ where each $e_j : V_j \to V_{j+1}$ is of maximal rank. For example, we can choose a basis $B_j = \{v^j_1, \dots, v^j_{d_j}\}$ for each $V_j$ and let $e_j$ be the one with matrix $\begin{pmatrix}
		\Id_{d_{j+1}} & 0
	\end{pmatrix}$ or $\begin{pmatrix}
		\Id_{d_{j}} \\
		0
	\end{pmatrix}$, depending on whether $d_j \le d_{j+1}$. In order to complete $e$ to an $\liesl_2$-triple, let $\{u_0,\dots, u_{s-1}\}$ be a basis for a Jordan block (that is, $e(u_k) = u_{k+1}$ for $k < s-1$, $e(u_s) = 0$, and $\Ima(e) \cap \gen{u_0} = 0$), and define a traceless linear map $h$ by $h(u_j) = -(s-1-2j)u_j$. Doing this on each Jordan block gives $h \in \mathfrak g$ with $[h,e] = 2e$ and the remaining $f$ is defined as $f(u_j) = j(s-j)u_{j-1}$. Moreover, notice that by definition of $e$, we can partition the basis $\bigcup_{j = 0}^{m-1}B_j$ into Jordan blocks. This means that the resulting $h$ given by the previous method is actually in $\mathfrak g_0$. As $\zeta$ is fixed by the action of $G_0$, the JM-regular cases are precisely whenever the $h$ constructed above (depending only on the dimensions $d_j$) equals $2\zeta$. An example of JM-regular case is given by partitions satisfying $d_j = d_{m-j-1}$ for $j \in \{0,\dots,c\}$ and $d_0 \le d_1 \le \dots \le d_c$, where $c = \lfloor\frac{m-1}{2}\rfloor$. 
\end{example}

\subsection{Toledo character}\label{toledochar}

Let $\lieg = \bigoplus_{j=1-m}^{m-1}\lieg_j$ be a $\Z$-graded semisimple complex Lie algebra with $\lieg_{m-1} \neq 0$. Let $B$ be an invariant bilinear form on $\lieg$ (such as the Killing form) and let $\liet \subseteq \lieg$ be a Cartan subalgebra. We denote by $B^*$ the dual form on $\liet^*$. After choosing roots $\Delta \subseteq \liet^*$ for $\lieg$, we may assume (up to inner automorphism of $\lieg$) that each root space is contained in some $\lieg_j$ of the $\Z$-grading \cite[Section 2.3]{biquard_arakelov-milnor_2021}. Let $\gamma \in \lie t$ be the longest root such that $\lie g_{\gamma} \subseteq \lie g_1$. We have the following definition from \cite[Section 3.1]{biquard_arakelov-milnor_2021}.

\begin{definition}
	The \textbf{Toledo character} $\chi_T : \lieg_0 \to \mathbb C$ is defined by
	$$\chi_T(x) = B(\zeta, x)B^*(\gamma, \gamma).$$
\end{definition} 

This is indeed a character, as $B(\zeta, [x,y]) = -B([x,\zeta],y) = B(0,y) = 0$. The constant factor $B^*(\gamma, \gamma)$ ensures that the definition does not depend on the choice of invariant bilinear form. 

Using the Toledo character, we can associate a number to each element of $\lie g_1$ depending on its orbit.

\begin{definition}
	Let $e \in \lie g_1$ and $(h,e,f)$ an $\lie{sl}_2$-triple with $h \in \lie g_0$. We define the \textbf{Toledo rank of $e$} by
	$$\rank\nolimits_T(e) := \frac{1}{2}\chi_T(h).$$ 
\end{definition}

This number is indeed independent of the representative of a given $G_0$-orbit: if $e,e' \in \mathfrak g_1$ belong to the same orbit and $h,h'$ are the corresponding elements in $\mathfrak g_0$, by \cite[Proposition 2.19]{biquard_arakelov-milnor_2021} we get that there exists $g \in G_0$ such that $\Ad_g h = h'$. Then, by $\Ad$-invariance of $B$ as well as the fact that $\Ad_g\zeta = \zeta$, we have $\chi_T(h) = B(\zeta, h)B^*(\gamma,\gamma) = B(\Ad_g\zeta, \Ad_g h)B^*(\gamma,\gamma) = \chi_T(h')$. Moreover, by \cite[Proposition 3.16]{biquard_arakelov-milnor_2021}, if $e' \in \Omega$ is an element of the open orbit, we have
$$0 \le \rank\nolimits_T(e) \le \rank\nolimits_T(e'),$$
\noindent with the second inequality becoming an equality if and only if $e \in \Omega$. In other words, the maximum value of the rank is given precisely by elements of the open orbit.

\begin{definition}
	We define the \textbf{Toledo rank of $(G_0,\mathfrak g_1)$} to be $\rank_T(G_0,\mathfrak g_1) := \rank_T(e)$ for any $e \in \Omega$.
\end{definition}

The Toledo rank is related to the sectional curvature of certain metric on a \textit{period domain}, a fibration over a Riemannian symmetric space, naturally associated to $(G_0,\lieg_1)$ \cite[Section 3.3]{biquard_arakelov-milnor_2021}.

We conclude with a definition and a result that will be needed in a future proof.

\begin{definition}
	Let $\rho : G \to \GL(V)$ be a prehomogeneous vector space and $\chi : G \to \mathbb C^*$ a character. A non-constant rational function $F: V \to \mathbb C$ is called a \textbf{relative invariant} for $\chi$ if, for all $g \in G$ and $v \in V$, we have
	$$F(\rho(g) \cdot v) = \chi(g)F(v).$$ 
\end{definition}

\begin{lemma}\label{relinv}
	Suppose that $(G_0, \mathfrak g_1)$ is a JM-regular Vinberg $\C^*$-pair. Then, there exists $q \in \mathbb Z_{>0}$ such that $q\chi_T$ lifts to a character $\tilde{\chi}_T : G_0 \to \mathbb C^*$ having a relative invariant $F: \lieg_1 \to \C$ of degree $q\rank_T(G_0,\mathfrak g_1)$.
\end{lemma}   

Proof. \cite[Proposition 2.8]{biquard_arakelov-milnor_2021} \qed

\subsection{$\Z/m\Z$-gradings of semisimple Lie algebras}\label{zmgradings}

Now we focus on $\Z/m\Z$-gradings on semisimple Lie algebras.

\begin{definition}
	Let $m \in \mathbb Z_{>0}$ and let $\mathfrak g$ be a semisimple complex Lie algebra. A \textbf{$\mathbb Z/m\mathbb Z$}-grading of $\mathfrak g$ is a decomposition as a direct sum of vector subspaces	\[\mathfrak g = \bigoplus_{j \in \mathbb Z/m\mathbb Z} \bar{\lieg}_j,\]
	\noindent such that $[\bar{\lieg}_j, \bar{\lieg}_k] \subseteq \bar{\lieg}_{j+k}$. 
\end{definition}

We will reserve the notation $\bar{\lieg}_j$ for each of the graded pieces on a $\Z/m\Z$-graded Lie algebra $\lieg$, in order to avoid confusion with the pieces of a $\Z$-grading.

\begin{proposition}\label{zmgradingmap}
	There is a correspondence between $\mathbb Z/m\mathbb Z$-gradings of $\mathfrak g$ and group homomorphisms $\gamma: \mu_m \to \aut(\mathfrak g)$, where $\mu_m = \{z \in \mathbb C^* : z^m = 1\}$ is the subgroup of $m$-th roots of unity. These homomorphisms are also in correspondence with order $m$ automorphisms $\theta \in \aut_m(\mathfrak g)$.
\end{proposition}

\begin{proof}
	Fix $\zeta \in \mathbb C^*$ a primitive $m$-th root of unity. From a $\mathbb Z/m \mathbb Z$-grading $\mathfrak g = \bigoplus_{j \in \mathbb Z/m\mathbb Z} \bar{\lieg}_j$, we obtain an order $m$ automorphism $\theta$ by the rule $\theta|_{\bar{\lieg}_j} \equiv \zeta^j\Id_{\bar{\lieg}_j}$. This is well defined and of order $m$, both because $\zeta^m = 1$. It is compatible with the bracket: if $X \in \bar{\lieg}_j$ and $Y \in \bar{\lieg}_k$, we have $\theta[X,Y] = \zeta^{j+k}[X,Y] = [\zeta^jX, \zeta^kY] = [\theta X, \theta Y]$.
	
	Conversely, given $\theta \in \aut_m(\mathfrak g)$, we obtain a $\mathbb Z/m\mathbb Z$-grading by taking the eigenspace decomposition, that is, setting $\bar{\lieg}_j = \{X \in \mathfrak g : \theta X = \zeta^j X\}$. The fact that it is a $\mathbb Z/m\mathbb Z$-grading comes from the compatibility of $\theta$ with the Lie bracket. 
	
	The last assertion is a consequence of $\mu_m$ being the cyclic group of order $m$. 
\end{proof}

Exactly as in the case of $\Z$-gradings, we have that $\bar{\lieg}_0 \subseteq \lieg$ is a Lie subalgebra and its corresponding connected subgroup $G_0 \le G$, which is reductive, has induced representations $\rho : G_0 \to \GL(\bar{\lieg}_j)$. These representations were studied by Vinberg \cite{vinberg_graded}.

\begin{definition}
	The pairs $(G_0,\bar{\lieg}_j)$ are called \textbf{Vinberg $\theta$-pairs}.
\end{definition}

\begin{remark}
	As in Remark \ref{regrading}, we may assume that $j=1$ by considering the subalgebra $\lieg' := \bigoplus_{k \in \Z/m'\Z} \bar{\lieg}_{kj}$, which is a graded subalgebra for $m' := \frac{m}{(m,k)}$ whose first piece is $\bar{\lieg}_j$.
\end{remark}

\begin{remark}
	In the case that the automorphism $\theta \in \aut_m(\lieg)$ giving the grading (cf. Proposition \ref{zmgradingmap}) comes from an automorphism of the group $\theta \in \Aut_m(G)$, we also get representations of any closed subgroup of $G_\theta := N_G(G^\theta)$ (here $G^\theta \le G$ is the subgroup of fixed points) containing $G_0$, since they also have $\bar{\lieg}_0$ as Lie algebra. This includes the subgroup $G^\theta$ itself. These groups may not be connected. In the case where $G$ is simply connected, we have $G^\theta = G_0$.
\end{remark}

We now list some examples of Vinberg $\theta$-pairs.

\begin{example}[Adjoint representation]\label{cyclicm1}
	The case of $m=1$ corresponds to the trivial $\Z/\Z$-grading $\lieg = \bar{\lieg}_0$ given by $\theta = \Id_\lieg \in \Aut_1(\lieg)$. In this case the Vinberg pair is $(G,\lieg)$, that is, the adjoint representation.
\end{example}

\begin{example}[Symmetric pairs]\label{cyclicm2}
	The case $m=2$ corresponds to symmetric pairs (see also Example \ref{realhermitian}). Recall that if $G^\R/H^\R$ is a symmetric space, we get a complexified Cartan decomposition $\lieg = \lieh \oplus \liem$. This decomposition is a $\Z/2\Z$-grading of $\lieg$ with $\bar{\lieg}_0 := \lieh$ and $\bar{\lieg}_1 := \liem$.
\end{example}

\begin{example}[Cyclic quiver representations]\label{cyclicquiver}
	As in the linear quiver representation case of Example \ref{linear_quiver}, we take $G = \SL_n(\mathbb C)$, seen as transformations of some $n$-dimensional vector space $V$. We once again fix a direct sum decomposition in $m$ pieces, $V = V_0 \oplus \dots \oplus V_{m-1}$, where $\dim V_i = d_i$. Then, we have an order $m$ automorphism $\theta \in \aut_m(G)$ defined by $\theta(g) = I_{d_0,\dots,d_{m-1}}\cdot g \cdot I_{d_0,\dots,d_{m-1}}^{-1}$, where the operator $I_{d_0,\dots,d_{m-1}} \in \GL_n(\mathbb C)$ is defined by $I_{d_0,\dots,d_{m-1}}|_{V_j} \equiv \zeta_m^j\Id_{V_j}$ for $\zeta_m \in \mathbb C^*$ a primitive $m$-th root of unity. Then, the Lie algebra $\mathfrak g = \mathfrak{sl}_n\mathbb C$ gets a $\mathbb Z/m\mathbb Z$ grading which is
	$$\bar{\lieg}_k := \left(\bigoplus_{j \in \mathbb Z/m\mathbb Z}\Hom(V_j, V_{j+k})\right)_0,$$
	
	\noindent where the subscript zero means taking the subspace of traceless endomorphisms, and again is only meaningful for $\bar{\lieg}_0$.
	
	The associated Vinberg $\theta$-pair is $(G_0,\bar{\lieg}_1)$, where
	$$G_0 = S(\GL\nolimits_{d_0}(\mathbb C) \times \dots \times \GL\nolimits_{d_{m-1}}(\mathbb C)),$$
	$$\bar{\lieg}_1 = \bigoplus_{j \in \mathbb Z/m\mathbb Z}\Hom(V_j, V_{j+1}).$$
	
	The space $\bar{\lieg}_1$ consists of representations of a quiver with $m$ vertices and arrows $i \mapsto i+1$ for $i \in \mathbb Z/m\mathbb Z$ (a \emph{cyclic quiver}) where we put $V_i$ on the $i$-th vertex: 
	\[\begin{tikzcd}
		{V_0} && {V_1} && \dots && {V_{m-1}}
		\arrow["{f_0}", from=1-1, to=1-3]
		\arrow["{f_1}", from=1-3, to=1-5]
		\arrow["{f_{m-2}}", from=1-5, to=1-7]
		\arrow["{f_{m-1}}", curve={height=-24pt}, from=1-7, to=1-1]
	\end{tikzcd}.\]	
\end{example}

\subsection{$\Z/m\Z$-gradings associated to $\Z$-gradings}\label{bothgradings} In this subsection we explore the relation between $\Z$-gradings and $\Z/m\Z$-gradings.

Given any $\Z$-grading $\lieg = \bigoplus_{j \in \Z} \lieg_j$, it induces an \textbf{associated $\Z/m\Z$-grading} by projecting the indices:
\[\bar{\lieg}_j := \bigoplus_{k \equiv j \mod m}\lieg_k.\]
The construction is equivalently explained as follows: the data of a $\Z$-grading is equivalent to specifying a Lie group homomorphism $\gamma: \C^* \to \Aut(\lieg)$ by letting $\lambda \in \C^*$ act as scaling by $\lambda^j$ on $\lieg_j$. Then, the associated $\Z/m\Z$-grading is $\gamma(\zeta_m) \in \Aut_m(\lieg)$, where $\zeta_m = \exp(2\pi i / m)$ is a primitive $m$-th root of unity (see Proposition \ref{zmgradingmap} for the equivalence between cyclic automorphisms and gradings). 

We will be interested in $\Z/m\Z$-gradings arising from $\Z$-gradings.

\begin{proposition}\label{innerequalsspecial}
	Let $\theta \in \Aut_m(\lieg)$. The corresponding $\Z/m\Z$-grading is associated to a $\Z$-grading if and only if $\theta \in \Int(\lieg)$.
\end{proposition}

\begin{proof}
	If the $\Z/m\Z$ grading is associated to a $\Z$-grading, there is a Lie group homomorphism $\gamma : \C^* \to \Aut(\lieg)$ such that $\gamma(\zeta_m) = \theta$ as explained in the paragraph above. By continuity, this means that $\theta \in \Aut(\lieg)_0 = \Int(\lieg)$.
	
	Conversely, by the classification of inner order $m$ automorphisms of $\lieg$ (see \cite[Theorem 3.3.11]{onivin}), there exists a Cartan subalgebra $\liet \subseteq \lieg$ and simple roots $\Pi \subseteq \Delta := \Delta(\lieg, \liet)$ such that:
	
	\begin{itemize}
		\item Each root space $\lieg_\alpha$, where $\alpha \in \Delta$, belongs to a graded piece $\bar{\lieg}_{\bar{p}_\alpha}$, $\bar{p}_\alpha \in \Z/m\Z$.
		\item We can choose representatives $\{p_j\}_{j=0}^r \subseteq \Z_{\ge 0}$ for the indices of the pieces for the simple roots $\Pi = \{\alpha_1,\dots,\alpha_r\}$ as well as for the lowest root $\alpha_0 = -\sum_{k=1}^rn_k\alpha_k$ (where $\delta := \sum_{k=1}^rn_k\alpha_k$ is the highest root), i.e. a labelling of the affine Dynkin diagram of $\lieg$, such that 
		\[m = \sum_{k=0}^rn_kp_k,\]
		(where we take $n_0 = 1$) and, for any root $\alpha = \sum_{k=0}^rn_k^\alpha\alpha_k \in \Delta$ written in the unique coordinates $n_k^\alpha$ with $0 \le n_k^\alpha \le n_k$, we can recover
		\[\bar{p}_{\alpha} = \overline{\sum_{k=0}^rn_k^\alpha p_k}.\]
	\end{itemize} 
	
	Given the above data from the classification, the $\Z/m\Z$-grading of $\theta$ is associated to the $\Z$-grading uniquely determined by setting the degree of the simple root spaces $\lieg_{\alpha_k}$ to be $p_k$. 
\end{proof}

Finally, the last step to relate the theory of Vinberg $\C^*$-pairs to that of Vinberg $\theta$-pairs is to ensure that the resulting group $G_0$ is the same in both cases. In other words, we will restrict the attention to $\Z/m\Z$-gradings which arise from $\Z$-gradings such that $\bar{\lieg}_0 = \lieg_0$. This motivates the notation used at the beginning of Section \ref{zgradings}, where $m$ was used to denote the smallest positive integer such that $|j| \ge m$ implies $\lieg_j = 0$. 

In short, we consider $\Z/m\Z$-gradings arising from $\Z$-gradings of the form
\[\lieg = \lieg_{1-m} \oplus \dots \oplus \lieg_{m-1}.\]
In this case $\bar{\lieg}_0 = \lieg_0$ and $\bar{\lieg}_j = \lieg_j \oplus \lieg_{1-j}$ for $j \in \{1,\dots,m-1\}$.

In terms of the classifying data introduced in the proof of Proposition \ref{innerequalsspecial}, the obtained $\Z$-grading is of the desired form if \[\left|-n_0^\alpha m + \sum_{k=0}^rn_k^\alpha p_k\right| \le m-1\]
\noindent for every $0 \le n_k^\alpha \le n_k$, where $\delta = \sum_{k=1}^rn_k\alpha_k = -\alpha_0$ is the highest root. This is because the left hand side of the previous expression is the degree of the root $\alpha = \sum_{k=0}^rn_k^\alpha\alpha_k$ in the associated $\Z$-grading. Recalling that $\sum_{k=0}^rn_kp_k = m$, the condition is satisfied if and only if $p_0 > 0$.

Furthermore, if we apply an automorphism of the affine Dynkin diagram, which permutes the labels $\{p_j\}$, the resulting $\Z/m\Z$-grading is equivalent to the starting one by conjugation. However, the associated $\Z$-grading obtained via this new set of labels need not be equivalent and could be of our desired form even if the starting one was not. We summarise in the following proposition.

\begin{proposition}
	If $\theta \in \Int_m(\lieg)$ is an inner order $m$ automorphism of $\lieg$, classified by the labels $\{p_j\}$ on the affine Dynkin diagram, and if there exists an automorphism of the diagram that maps a nonzero label $p_j > 0$ to the lowest root $\alpha_0$, then the $\Z/m\Z$-grading given by $\theta$ lifts to a $\Z$-grading of the desired form
	\[\lieg = \lieg_{1-m} \oplus \dots \oplus \lieg_0 \oplus \dots \oplus \lieg_{m-1}.\]
\end{proposition}

\begin{remark}
	The condition can be rephrased as: there is a nonzero label $p_j > 0$ corresponding to a simple or lowest root $\alpha_j$ such that $n_j = 1$ (recall that, for $j > 0$, $n_j$ is the coefficient of $\alpha_j$ in the highest root and $n_0 = 1$).
\end{remark}

Since the automorphism group of the affine Dynkin diagram in type $A$ acts transitively, we deduce the following corollary.

\begin{corollary}
	Every $\theta \in \Int_m(\liesl_n\C)$ lifts to a $\Z$-grading of the desired form.
\end{corollary}

In particular, the cyclic quiver grading from Example \ref{cyclic_quiver} comes from a $\Z$-grading of the desired form, which is just the one from Example \ref{linear_quiver}.

\begin{remark}
	In the case of symmetric pairs from Example \ref{cyclicm2}, the ones coming from a $\Z$-grading of the form $\lieg = \lieg_{-1} \oplus \lieg_0 \oplus \lieg_1$ are precisely the Hermitian ones from example \ref{realhermitian}.
\end{remark}

\begin{remark}
	We can always go in the other direction: for any given $\Z$-grading of $\lieg$, since the latter is finite dimensional, we can find a value of $m$ such that
	\[\lieg = \lieg_{1-m} \oplus \dots \oplus \lieg_0 \oplus \dots \oplus \lieg_{m-1}.\]
	Then it is possible to consider the induced $\Z/m\Z$-grading, which satisfies $\bar{\lieg}_0 = \lieg_0$ and $\bar{\lieg}_1 = \lieg_1 \oplus \lieg_{1-m}$. All the theory developed in this paper allows to study the cyclic Higgs bundles determined by any automorphism $\theta \in \Aut_m(G)$ lifting this $\Z/m\Z$-grading of $\lieg$. 
\end{remark}

\section{Higgs bundles and the Toledo invariant}\label{higgstoledo}

Throughout this section, $X$ is a compact Riemann surface of genus $g \ge 2$ and $K_X$ is its canonical line bundle.

\subsection{Higgs bundles and stability}\label{higgsbundles}

We recall the fundamentals of the theory of Higgs pairs associated to an arbitrary representation of a reductive group. Thus, in this Section \ref{higgsbundles}, $G$ denotes a complex reductive Lie group with Lie algebra $\lieg$ and $\rho : G \to GL(V)$ a holomorphic representation. Let also $L$ be a line bundle over $X$. 

\begin{definition}\label{higgspair}
	An \textbf{$L$-twisted $(G,V)$-Higgs pair} is a pair $(E,\varphi)$, where $E$ is a holomorphic principal $G$-bundle on $X$, and $\varphi \in H^0(X, E(V) \otimes L)$, where $E(V) := E \times_G V$ is the associated vector bundle to $E$ via $\rho$. A \textbf{$(G,V)$-Higgs pair} is a $K_X$-twisted $(G,V)$-Higgs pair.
\end{definition}

When $V = \lieg$ and $\rho = \Ad$, the resulting $(G,\lieg)$-Higgs pairs are called $G$-Higgs bundles. When $G^\R \le G$ is a real form with complexified maximal compact $H \le G$ and complexified Cartan decomposition $\lieg = \lieh \oplus \liem$, $V = \liem$, and $\rho : H \to \GL(\liem)$ is the restriction of the adjoint representation, the resulting $(H,\liem)$-Higgs pairs are called $G^\R$-Higgs bundles.

There exist notions of stability \cite{garcia-prada_hitchin-kobayashi_2012} (\cite{garcia-prada_hitchin_kobayashi_nonconnected_unpublished} for the nonconnected $G$ case) which allow to construct a moduli space of $(G,V)$-Higgs pairs. We now recall them. We need the following definition:

\begin{definition}
	Let $G$ be a complex reductive Lie group, $\hat{G} \le G$ a Lie subgroup, and $E$ a principal $G$-bundle. A \textbf{reduction of structure group} of $E$ to $\hat{G}$ is a holomorphic section $\sigma \in H^0(X,E(G/\hat{G}))$, where $E(G/\hat{G}) = E \times_{G} G/\hat{G}$.
\end{definition}

The natural map $E \to E(G/\hat{G})$ has a $\hat{G}$-bundle structure. From the previous definition, a reduction of the structure group is a map $\sigma : X \to E(G/\hat{G})$. Thus, it is possible to pull back the $\hat{G}$-bundle $E$ on $E(G/\hat{G})$ to get $E_{\sigma} := \sigma^*E$, a $\hat{G}$-bundle on $X$. Moreover, there is a canonical isomorphism $E_{\sigma} \times_{\hat{G}} G \simeq E$. The map $E_{\sigma} = \sigma^*E \to E$ induced by the pullback is injective and gives a holomorphic subvariety $E_{\sigma} \subseteq E$. 

Fix a maximal compact subgroup $K \le G$. Let $\mathfrak{k}$ be its Lie algebra, a real subalgebra of $\mathfrak g$. Define for $s \in i\mathfrak k$ the spaces
$$V_s^0 = \{v \in V : \forall t \in \R, \rho(e^{ts})(v) = v\}, \quad V_s = \{v \in V : \rho(e^{ts})(v) \text{ is bounded as } t \to \infty\}.$$ 
\noindent and the subgroups
$$L_s = \{g \in G : \Ad(g)(s)=s\}, \quad P_s = \{g \in G : e^{ts}ge^{-ts} \text{ is bounded as } t \to \infty\}.$$ 
These subgroups of $G$ correspond to the lie algebras $\mathfrak g_s^0$ and $\mathfrak g_s$ defined for the adjoint representation. We also define a character $\chi_s : \mathfrak g_s \to \mathbb C$ given by $\chi_s(x) = B(s,x)$, where $B$ is the Killing form on $\mathfrak g$. The subgroup $L_s$ acts on $V_s^0$, and the group $P_s$ acts on $V_s$ both via $\rho$.

Let $E$ be a principal $G$-bundle and let $\sigma \in H^0(X, E(G/P_s))$ be a reduction of $E$ to $P_s$. If a multiple $q \chi_s$ for some $q \in \Z_{>0}$ lifts to a character $\tilde{\chi}_s :P_s \to \mathbb C^*$, we define the \textbf{degree of the reduction} as
$$\deg E(\sigma, s) := \frac{1}{q}\deg (E_\sigma \times_{\tilde{\chi}_s} \mathbb C^*),$$
\noindent where the line bundle on the right hand side is the associated bundle to the $P_s$-bundle $E_{\sigma}$ via the character $\tilde{\chi}_s$. It is also possible to define this quantity when no multiple of the character lifts to the group, using differential geometric techniques, as follows: given a reduction $\sigma$ of $E$ to $P_s$, there is a further reduction $\sigma'$ to $K_s := K \cap L_s$, the maximal compact of $L_s$. Let $A$ be a connection on $E_{\sigma'}$ and consider its curvature $F_A \in \Omega^2(X, E_{\sigma'}(\mathfrak k_s))$. We have that $\chi_s(F_A) \in \Omega^2(X, i \R)$, and the degree is defined as
$$\deg E(\sigma,s) := \frac{i}{2\pi}\int_X\chi_s(F_A).$$

Now let $\mathfrak z$ be the center of $\mathfrak k$, so that $\mathfrak k = \mathfrak z \oplus \mathfrak k_{ss}$, where $\mathfrak k_{ss}$ is the semisimple part. Given the representation $\rho : G \to \GL(V)$, let $d\rho : \mathfrak g \to \mathfrak{gl}(V)$ be its differential, and define
$$\mathfrak k_{\rho} := \mathfrak k_{ss} \oplus \ker(d\rho|_{\mathfrak k})^\perp,$$
\noindent where the orthogonal is taken in $\mathfrak k$ with respect to the Killing form.

In order to deal with the nonconnected case, denote by $G^0 \le G$ the connected component of the identity, meaning that we have an extension $1 \to G^0 \to G \to \Gamma \to 1$ where we assume that $\Gamma = \pi_0(G)$ is finite. By results of \cite{barajas_non-connected_2023}, there exists an action $\theta : \Gamma \to \aut(G^0)$ and an homomorphism $c: \Gamma \times \Gamma \to Z(G^0)$ (which, in terms of group cohomology, is moreover a cocycle with respect to the action $\theta$), such that $G \simeq G^0 \times_{(\theta, c)} \Gamma$, the latter subscript meaning that the group operation is given by $(g_1, \gamma_1) \cdot (g_2,\gamma_2) = (g_1\theta(\gamma_1)(g_2)c(\gamma_1,\gamma_2),\gamma_1\gamma_2)$. The maximal compact $K$ can also be taken to be $\Gamma$-invariant. Thus it makes sense to consider the fixed points $\mathfrak z^{\Gamma} \subseteq \mathfrak z$ and $\mathfrak k^{\Gamma} \subseteq \mathfrak k$. 

We can now define stability (see \cite{ garcia-prada_hitchin_kobayashi_nonconnected_unpublished}).

\begin{definition}
	Fix a parameter $\alpha \in i\mathfrak z^{\Gamma}$. A $(G,V)$-Higgs pair $(E,\varphi)$ is:
	\begin{itemize}
		\item \textbf{$\alpha$-semistable}, if for any element $s \in i \mathfrak k^{\Gamma}$ and reduction $\sigma \in H^0(X, E(G/P_s))$ such that $\varphi \in H^0(E_\sigma(V_s) \otimes K_X)$, we have $\deg E(\sigma, s) \ge B(\alpha, s)$.
		\item \textbf{$\alpha$-stable}, if it is $\alpha$-semistable and, for any element $s \in i\mathfrak k^{\Gamma}_{\rho}$ and reduction $\sigma \in H^0(X, E(G/P_s))$ such that $\varphi \in H^0(E_\sigma(V_s) \otimes K_X)$, we have $\deg E(\sigma, s) > B(\alpha,s)$.
		\item \textbf{$\alpha$-polystable}, if it is $\alpha$-semistable and, for the $s \in i \mathfrak k^{\Gamma}$ and $\sigma \in H^0(X, E(G/P_s))$ such that $\varphi \in H^0(E_\sigma(V_s) \otimes K_X)$ and we have $\deg E(\sigma, s) = B(\alpha, s)$, there exists a reduction $\sigma' \in H^0(E_\sigma(P_s/L_s))$ of $E_\sigma$ to $L_s$ such that $\varphi \in H^0(E_{\sigma'}(V_s^0) \otimes K_X)$.
	\end{itemize}
\end{definition}

The \textbf{moduli space of $\alpha$-polystable $(G,V)$-Higgs pairs} over $X$ parametrises isomorphism classes of $\alpha$-polystable $(G,V)$-Higgs pairs. We denote it by $\mathcal M^\alpha(G,V)$. A construction as a geometric space via Geometric Invariant Theory is given by Schmitt in \cite{schmitt2008geometric}. When $\alpha = 0$, we simply refer to \textit{semistable, polystable} and \textit{stable} Higgs pairs, and denote the moduli space of polystable bundles as $\mathcal M(G,V)$. We have the following important particular cases discussed after Definition \ref{higgsbundles}: the moduli space of polystable $G$-Higgs bundles, which is denoted by $\mathcal M(G)$, and that of polystable $G^\R$-Higgs bundles, denoted by $\mathcal M(G^\R)$.

All the previous definitions of stability are identical for $L$-twisted $(G,V)$-Higgs pairs for an arbitrary line bundle $L$ over $X$, simply replacing $K_X$ with $L$ when appropriate. The resulting moduli space of $\alpha$-polystable $L$-twisted $(G,V)$-Higgs pairs will be denoted by $\mathcal M^\alpha_L(G,V)$.

\subsection{Cyclic Higgs bundles and the non-abelian Hodge correspondence}\label{chb}

We consider Higgs pairs for the representations studied in Section \ref{zgrad}. Due to this, from now on we adopt the notation in that section: let $G$ be a complex semisimple Lie group with Lie algebra $\lieg$ and $\theta \in \Aut_m(G)$ an automorphism of finite order $m \in \Z_{>0}$. Let $\zeta_m \in \C^*$ be a primitive $m$-th root of unity. Given a $G$-Higgs bundle $(E,\varphi)$, let $\theta(E) := E \times_\theta G$ and let $\theta(\varphi)$ be the induced section of $\theta(E)(\lieg) \otimes K_X$ via the induced $\theta \in \Aut(\lieg)$. There is an action of $\Z/m\Z$ in $\mathcal M(G)$ generated by
\[(E,\varphi) \mapsto (\theta(E), \zeta_m \cdot \theta(\varphi)).\]

Elements on the fixed point locus are called \textbf{cyclic (or $\theta$-cyclic) Higgs bundles}. They were described in \cite{garcia-prada_involutions_2019} as follows. Given $\theta' \in \Aut_m(G)$ in the same outer class as $\theta$ (i.e. $\theta'\theta^{-1} \in \Int(G)$, where $\Int(G)$ is the group of inner automorphisms of $G$), consider the Vinberg $\theta'$-pair $(G^{\theta'}, \bar{\lieg}_1)$ associated to the $\Z/m\Z$-grading given by $\theta'$. We use this representation to define moduli spaces $\mathcal M(G^{\theta'}, \bar{\lieg}_1)$.

Extension from $(G^{\theta'}, \bar{\lieg}_1)$ to $(G,\lieg)$ gives a map
\[\mathcal M(G^{\theta'}, \bar{\lieg}_1) \to \mathcal M(G).\] 
We denote its image by $\widetilde{\mathcal M}(G^{\theta'}, \bar{\lieg}_1)$. From \cite[Theorem 6.3]{garcia-prada_involutions_2019}, we have that for all possible choices of $\theta'$ (in the same outer class as $\theta$), the $\widetilde{\mathcal M}(G^{\theta'}, \bar{\lieg}_1)$ are $\theta$-cyclic. Moreover, the stable and simple $\theta$-cyclic Higgs bundles are all present in these images.

Recall from Section \ref{bothgradings} our situation of interest: we assume that there exists a $\Z$-grading $\lieg = \bigoplus_{j=1-m}^{m-1}\lieg_j$ with $\lieg_{m-1} \neq 0$ giving the $\Z/m\Z$-grading. We then obtain cyclic Higgs bundles by considering $\mathcal M(G^\theta, \lieg_1 \oplus \lieg_{1-m})$ as before, as well as $\mathcal M(G_0, \lieg_1 \oplus \lieg_{1-m})$ after extension of structure group to $G^\theta$.

\begin{example}
	If $G^\R \le G$ is a real form of Hermitian type (see Examples \ref{realhermitian} and \ref{cyclicm2}), we consider the corresponding complexified Cartan decomposition $\lieg = \lieh \oplus \liem$. This is a $\Z/2\Z$-grading of $\lieg$ obtained from a $\Z$-grading $\lieg = \lieg_{-1} \oplus \lieg_0 \oplus \lieg_1$, so that we obtain the corresponding $\theta \in \Aut_2(\lieg)$. Moreover, it comes from $\theta \in \Aut_2(G)$, obtained by letting $\sigma$ be the antiholomorphic involution of $G$ fixing $G^\R$, letting $\tau$ be an antiholomorphic involution for a maximal compact such that $\tau\sigma = \sigma\tau$, and setting $\theta := \tau\sigma$. The corresponding moduli space is precisely $\mathcal M(G^\theta, \lieg_{-1} \oplus \lieg_1) = \mathcal M(H, \liem) = \mathcal M(G^\R)$.
\end{example}

\begin{example}\label{cyclic_quiver}
	For the $\Z$-grading of Example \ref{linear_quiver}, recall from Example \ref{cyclic_quiver} that the associated $\theta \in \Aut_m(\lie{sl}(V))$ lifts to $\theta \in \Aut_m(\SL(V))$ defined by $\theta(g) = h\cdot g \cdot h^{-1}$, where the operator $h \in \GL_n(\mathbb C)$ is given by $h|_{V_j} \equiv \zeta_m^j\Id_{V_j}$ for $\zeta_m$ a primitive $m$-th root of unity. From the fact that $\SL(V)$ is a classical Lie group, we can view $(G_0, \bar{\lieg}_{1})$-Higgs bundles as pairs $(E,\varphi)$ comprised of a holomorphic vector bundle $E = E_0 \oplus \dots \oplus E_{m-1}$ with $\det E \simeq \mathcal O$ and $\rank E_j = d_j$, and $\varphi : E \to E \otimes K_X$ a holomorphic traceless endomorphism with $\varphi(E_j) \subseteq E_{j+1} \otimes K_X$, the indices taken in $\Z/m\Z$. When $m=2$, the resulting moduli space matches $\mathcal M(\SU(d_0,d_1))$ which also belongs to the previous example.
\end{example}

We conclude the section by describing in further detail how the ingredients appearing in the non-abelian Hodge correspondence particularise to our context. We start on the side of Higgs pairs by collecting the Hitchin--Kobayashi correspondence for the spaces $\mathcal M(G_0, \lieg_{1-m} \oplus \lieg_1)$ (more generally, the same correspondence works for arbitrary Vinberg $\theta$-pairs). This follows from the general correspondence \cite{garcia-prada_hitchin-kobayashi_2012,garcia-prada_hitchin_kobayashi_nonconnected_unpublished} as explained in \cite[Section 6.4]{bradlow_general_2021}. 

Let $\tau : \lieg \to \lieg$ be a conjugate-linear involution associated to a compact real form $K \le G$ with Lie algebra $\liek$ and such that $\tau(\bar{\lieg}_j) = \bar{\lieg}_{-j}$ (see \cite[Theorem 3.72]{wallach} for its existence). Let $L \to X$ be a line bundle and fix a Hermitian metric $h_L$ on it. This identifies $L$ with $L^*$. For an $L$-twisted $(G_0, \lieg_{1-m} \oplus \lieg_1)$-Higgs pair $(E,\varphi)$ and a metric $h \in H^0(X, E(G/K))$, we get an involution $\tau_h : E_h(\lieg) \times L \to E_h(\lieg) \otimes L^*$ from $\tau$ and the identification $L \simeq L^*$. Since $\varphi \in H^0(X,E_h(\lieg_{1-m} \oplus \lieg_1) \otimes L)$, we have $[\varphi, -\tau_h(\varphi)] \in H^0(X, E_h(\liek_0))$, for $\liek_0 = \liek \cap \bar{\lieg}_0$. Let $\liez_0$ be the centre of $\liek_0$ and $\Gamma$ as in Section \ref{higgsbundles}. Let $\omega \in \Omega^{1,1}(X)$ be a Kähler form on $X$, and $F_h \in \Omega^{1,1}(E_h(\liek_0))$ the curvature of the Chern connection for $h$.

\begin{proposition}[{Hitchin--Kobayashi correspondence for $L$-twisted $(G_0, \lieg_{1-m} \oplus \lieg_1)$-Higgs pairs \cite[Theorem 2.24]{garcia-prada_hitchin-kobayashi_2012}, \cite[Section 6.4]{bradlow_general_2021}}]\label{hitkob}
	Let $\alpha \in i\liez_0^\Gamma$ be a stability parameter. An $L$-twisted $(G_0, \lieg_{1-m} \oplus \lieg_1)$-Higgs pair is $\alpha$-polystable if and only if there exists a metric $h$ on $E$ such that the equality
	\[F_h + [\varphi,-\tau_h(\varphi)]\omega = -i\alpha\omega\]
	\noindent of elements in $\Omega^{1,1}(X,E_h(\liek_0))$ holds.
\end{proposition} 

Now we assume that $L = K_X$ in order to explore the remaining parts of the non-abelian Hodge correspondence. Recall that the existence of the metric $h$ above corresponds to the existence of a harmonic metric $h'$ on a (smooth) flat $G$-bundle $\mathbb E$. This, by a theorem of Corlette and Donaldson \cite{corlette_flat_1988, donaldson_twisted_1987}, corresponds to a completely reducible representation $\rho : \pi_1(X) \to G$ up to conjugation by $G$. 

We will explore some properties of the harmonic metric $h'$ that arises from a polystable $(G_0,\lieg_{1-m} \oplus \lieg_1)$-Higgs pair. Letting $K \le G$ be, as above, a compact real form of $G$ compatible with the grading, recall that a harmonic metric $h'$ on a flat $G$-bundle $\mathbb E$ is the same as a $\pi_1(X)$-equivariant harmonic map $h' : \tilde{X} \to G/K$, where $\tilde{X}$ is the universal covering space of $X$. We have the following result from \cite[Theorem 4.3.4]{collier_finite_2016}.

\begin{proposition}\label{liftharm}
	Let $h' : \tilde{X} \to G/K$ be the harmonic map corresponding to a polystable $(G_0, \lieg_{1-m} \oplus \lieg_1)$-Higgs pair. Let $K_0 := K \cap G_0$ be the maximal compact subgroup of $K$. There exists an equivariant map $f : \tilde{X} \to G/K_0$ such that
	\[\begin{tikzcd}
		&& {G/K_0} \\
		{\tilde{X}} && {G/K}
		\arrow["{h'}"', from=2-1, to=2-3]
		\arrow[two heads, from=1-3, to=2-3]
		\arrow["f", from=2-1, to=1-3]
	\end{tikzcd}\]
	\noindent commutes. The map $f$ is harmonic and $df(T^{1,0}\tilde{X}) \subseteq G \times_{K_0} ((\lieg_{1-m} \oplus \lieg_1) \otimes \C) \subseteq T^\C G/K_0$.
\end{proposition}  

Now define an involution $\theta_{Hodge}: \lieg \to \lieg$ by $\theta_{Hodge}|_{\lieg_j} = (-1)^j\Id_{\lieg_j}$. The compact (conjugate-linear) involution $\tau : \lieg \to \lieg$ corresponding to $K$ commutes with $\theta_{Hodge}$ since it satisfies $\tau(\lieg_j) = \lieg_{-j}$ by \cite[Remark 3.11]{biquard_arakelov-milnor_2021}. Hence $\sigma_{Hodge} := \tau \circ \theta_{Hodge} : \lieg \to \lieg$ defines a real form. Moreover, if $\liet \subseteq \lieg$ is a Cartan subalgebra containing the grading element $\zeta$ (which exists because $\zeta$ is semisimple), we must have $\liet \subseteq \lieg_0$, so that $\theta_{Hodge}|_\liet = \Id_{\liet}$. This implies that $\sigma_{Hodge}$ is an inner automorphism and then the real form is said to be of \textit{Hodge type}.

Let $G^\R \le G$ be the real form of Hodge type corresponding to $\sigma_{Hodge}$ and $\lieg^\R$ its Lie algebra. We now recall from \cite[Section 3.2]{biquard_arakelov-milnor_2021} the notion of a \textit{period domain}. By construction we have $G^\R \cap G_0 = G_0 \cap K = K_0$, the compact real form of $G_0$ from before. The homogeneous space
$$D := G^\R/K_0$$
\noindent has as tangent bundle $TD = G^\R \times_{K_0} \mathfrak{q}^\R$ where $\lieg^\R = \liek_0 \oplus \mathfrak{q}^\R$ is an orthogonal decomposition, with complexification $\lieg = \lieg_0 \oplus \left(\bigoplus_{j \neq 0}\lieg_j\right)$. Then $D$ has a natural complex structure, obtained by splitting
$$T^\C D = G^\R \times_{K_0} \bigoplus_{j \neq 0}\lieg_j$$
\noindent into the two pieces
$$T^{1,0}D = G^\R \times_{K_0}\bigoplus_{j > 0}\lieg_j,\quad T^{0,1}D = G^\R \times_{K_0}\bigoplus_{j < 0}\lieg_j.$$

\begin{remark}
	The Toledo character from Section \ref{toledochar} is related to geometric properties of metrics of minimal holomorphic sectional curvature on the period domain $D$. See \cite[Section 3.3]{biquard_arakelov-milnor_2021} for details on this.
\end{remark}

Finally, in the case where $m$ is even, taking a pair $(E,\varphi) \in \mathcal M(G_0, \lieg_{1-m} \oplus \lieg_1)$ and extending its structure group to $G^{\theta_{Hodge}}$ results in a $G^\R$-Higgs bundle, since $\theta_{Hodge}$ acts by $-1$ on both $\lieg_{1-m}$ and $\lieg_1$. At the level of the non-abelian Hodge correspondence this has the consequence that the diagram of Proposition \ref{liftharm} factors through the corresponding homogeneous spaces for $G^\R$, resulting in
\[\begin{tikzcd}
	&& D & {G/K_0} \\
	{\tilde{X}} && {G^\R/K^\R} & {G/K}
	\arrow["{h'}"', from=2-1, to=2-3]
	\arrow["f", from=2-1, to=1-3]
	\arrow[two heads, from=1-3, to=2-3]
	\arrow[hook, from=2-3, to=2-4]
	\arrow[hook, from=1-3, to=1-4]
	\arrow[two heads, from=1-4, to=2-4]
\end{tikzcd},\]
\noindent where $K^\R = K \cap G^\R$ is the maximal compact subgroup of $G^\R$. We remark that the map $f$ satisfies $df(T^{1,0}\tilde{X}) \subseteq G^\R \times_{K_0} (\lieg_{1-m} \oplus \lieg_1)$ so that it is not holomorphic. The case where the Higgs field takes values only in $\lieg_1$, that is, when the projection to $E(\lieg_{1-m}) \otimes K_X$ vanishes, corresponds to the case where $f$ is holomorphic. Such a holomorphic $f$ is known as a \textbf{variation of Hodge structure}. Higgs bundles corresponding to variations of Hodge structure have been studied extensively, for example in \cite{biquard_arakelov-milnor_2021, simpson_higgs_1992}.

\subsection{The Toledo invariant}\label{toledo}

We work with a complex semisimple Lie group $G$ with $\Z$-graded Lie algebra $\lieg = \bigoplus_{j=1-m}^{m-1}\lieg_j$ (such that $\lieg_{m-1} \neq 0$), with grading element $\zeta \in \lieg_0$ inducing $\theta \in \Aut_m(G)$ as in Section \ref{chb}. Let $\chi_T: \lieg_0 \to \C$ be the Toledo character from Section \ref{toledochar}, associated to the Vinberg $\C^*$-pair $(G_0,\lieg_1)$. As in the definition of stability of Higgs pairs in Section \ref{higgsbundles}, we can define the degree of $E$ associated to $\chi_T$ by selecting $q \in \mathbb \Z_{>0}$ such that $q\chi_T$ lifts to a character $\tilde{\chi} : G_0 \to \mathbb C^*$, and setting
\[\deg_{\chi_T}(E) := \frac{1}{q}\deg (E \times_{\tilde{\chi}_T} \mathbb C^*).\]

\begin{definition}
	Let $(E,\varphi)$ be a $(G_0, \lieg_{1-m} \oplus \lieg_1)$-Higgs pair. We define the \textbf{Toledo invariant} of $(E,\varphi)$ by
	\[\tau(E,\varphi) := \deg_{\chi_T}(E).\]
\end{definition}

\begin{remark}\label{dual}
	The invariant has been defined by using the prehomogeneous vector space $(G_0,\mathfrak g_1)$. However, since the Higgs field takes values in $\mathfrak g_1 \oplus \mathfrak g_{1-m}$, it makes sense to consider what happens if we try to use instead the prehomogeneous vector space $(G_0,\mathfrak g_{1-m})$. Recall from Remark \ref{regrading} that this space is of the form $(G_0,\mathfrak g_1')$ for the graded subalgebra $\mathfrak g' = \mathfrak g_{m-1} \oplus \mathfrak g_0 \oplus \mathfrak g_{1-m}$. The corresponding grading element is $\zeta' = \frac{\zeta}{1-m}$. We also need to select a new longest root $\gamma'$, now with the condition that the root space $\mathfrak g_{\gamma'}$  belongs in $\mathfrak g_{1-m}$. We remark that due to the classification of gradings in Section \ref{zmgradings}, $\gamma'$ is always a longest root of $\lieg$ (for example, it is a long root if $\lieg$ is simple). The corresponding Toledo character is $\chi_T'(x) = B(\zeta',x)B^*(\gamma',\gamma') = \frac{1}{1-m}\frac{B^*(\gamma',\gamma')}{B^*(\gamma, \gamma)}\chi_T(x)$. Consequently, we obtain an alternative Toledo invariant
	$$\tau'(E,\varphi) := \frac{1}{1-m}\frac{B^*(\gamma',\gamma')}{B^*(\gamma, \gamma)}\tau(E,\varphi).$$
	Both invariants differ by a negative constant. If $\lieg$ is simple and of simply laced Dynkin diagram (types $A,D,E$), all root lengths are the same and the constant is just $\frac{1}{1-m}$. However, if $\lieg$ is simple but not simply laced, the constant may acquire an extra factor depending on whether every root space contained in $\lieg_1$ corresponds to a short root.
\end{remark}

\begin{example}\label{exampletoledocyclic}
	We will compute the value of the Toledo invariant for the $(G_0, \lieg_{1-m} \oplus \lieg_1)$-Higgs pair of Example \ref{cyclic_quiver}. Recall from said example that these can be seen as vector bundles $E = E_0 \oplus \dots \oplus E_{m-1}$ with $\det E \simeq \mathcal O$, and with a Higgs field $\varphi \in H^0(X, \End(E) \otimes K_X)$ such that $\varphi(E_k) \subseteq E_{k+1} \otimes K_X$, the indices taken in $\mathbb Z/m\mathbb Z$. We first compute the Toledo character. Recall that in this case the grading element for the prehomogeneous vector space $(G_0,\mathfrak g_1)$, using the notation of Example \ref{linear_quiver}, is $\zeta \in \mathfrak g_0 = (\bigoplus_{j = 0}^{m-1}\End(V_j))_0$ given by $\zeta|_{V_j} = (j-\alpha)\Id|_{V_j}$. In $\SL_n(\mathbb C)$ we can choose the invariant form $B(X,Y) = \tr(XY)$. We get, for $x = (f_0,\dots,f_{m-1}) \in \mathfrak g_0$, that
	$$\chi_T(x) = B(\zeta, x)B^*(\gamma, \gamma) = 2\sum_{j=0}^{m-1}(j-\alpha)\tr(f_j).$$	
	A multiple $q\chi_T$ lifts to a character of the group $G_0$ if $2q(\alpha-j)$ is integral for all $j \in \{0,\dots,m-1\}$, resulting in 
	$$\tilde{\chi}(g) = \prod_{j=0}^{m-1}\det(g_j)^{2q(j-\alpha)},$$
	\noindent for $g = (g_0,\dots,g_{m-1}) \in G_0$. Such $q$ exists as $\alpha$ is rational. Then, we have the line bundle
	$$E\times_{\tilde{\chi}} \mathbb C^* = \bigotimes_{j=0}^{m-1}\det(E)^{\otimes 2q(j-\alpha)},$$
	\noindent so that, finally,
	$$\tau(E,\varphi) = \frac{1}{q} \deg (E\times_{\tilde{\chi}} \mathbb C^*) = 2\sum_{j=0}^{m-1}(j-\alpha)\deg E_j.$$
	It can be computed that $\alpha = \frac{\sum_{j=0}^{m-1}jd_j}{\sum_{j=0}^{m-1}d_j}$, and we have that $d_j = \rank E_j$, so we see that the resulting invariant depends on the degrees and ranks of each piece $E_j$.
	
	In the case $m=2$, if we let $(p,q)$ and $(a,b)$ be the ranks and degrees, respectively, of each of the two pieces $E_0$ and $E_1$, the previous formula reads
	$$\tau(E,\varphi) = 2\frac{pb-qa}{p+q},$$
	\noindent which is the Toledo invariant for $\U(p,q)$-Higgs bundles defined originally in \cite{bradlow_surface_2003}. Notice that in this case the considered bundle reduces to a $\SU(p,q)$-Higgs bundle and we have $b = -a$.
\end{example}

\subsection{Arakelov--Milnor--Wood inequality}\label{amwsec}
We work with the same context and notation as in Section \ref{toledo}. Given a $(G_0, \lieg_{1-m} \oplus \lieg_1)$-Higgs pair $(E, \varphi)$, we decompose \[\varphi = \varphi^+ + \varphi^-\]
\noindent where $\varphi^+ \in H^0(X, E(\lieg_1)\otimes K_X)$ and $\varphi^- \in H^0(X, E(\lieg_{1-m}) \otimes K_X)$.

\begin{remark}\label{genericorbit}
	Recall from Proposition \ref{finorbits} that there are finitely many $G_0$-orbits in $\lieg_1$ (resp. in $\lieg_{1-m}$). This implies that there exists a $G_0$-orbit where $\varphi^+$ (resp. $\varphi^-$) takes values generically.
\end{remark}

\begin{definition}
	The \textbf{Toledo rank} of $\varphi^+$ (resp. of $\varphi^-$) is defined as \[\rank_T(\varphi^+) := \rank_T(e)\] 
	\noindent for any element $e \in \lieg_1$ (resp. $e \in \lieg_{1-m}$) belonging to the $G_0$-orbit taken generically by $\varphi^+$ (resp. by $\varphi^-$).
\end{definition}

Notice that the definition for the Toledo rank of $\varphi^-$ uses the Vinberg $\C^*$-pair $(G_0,\lieg_{1-m})$, which involves a Toledo character that differs by a constant from that of $(G_0,\lieg_1)$, as discussed in Remark \ref{dual}.

We can extend the bound for the Toledo invariant in \cite{biquard_arakelov-milnor_2021} corresponding to the case $\varphi^- \equiv 0$.

\begin{theorem}[Arakelov--Milnor--Wood inequality]\label{amw}
	Let $G$ be a complex semisimple Lie group with $\Z$-graded Lie algebra $\lieg = \bigoplus_{j=1-m}^{m-1}\lieg_j$ for some $m \in \mathbb Z_{>0}$ such that $\lieg_{m-1} \neq 0$ and with grading element $\zeta \in \lieg_0$. Let $\alpha := \lambda \zeta$ for $\lambda \in \R$. Let $\gamma$ be the longest root such that $\lieg_{\gamma} \subset \lieg_1$.
	
	If $(E,\varphi)$ is an $\alpha$-semistable $(G_0,\lieg_1 \oplus \lieg_{1-m})$-Higgs pair, the Toledo invariant $\tau(E,\varphi)$ satisfies the inequality
	$$-\tau_L \le \tau(E,\varphi),$$
	\noindent where
	\begin{align*}
		\tau_L &= \rank\nolimits_T(\varphi^+)(2g-2) - \lambda(B^*(\gamma,\gamma)B(\zeta, \zeta) - \rank\nolimits_T(\varphi^+)).
	\end{align*}
	
	Moreover, if $m=2$ or $\varphi^- = 0$, the upper bound
	$$\tau(E,\varphi) \le \tau_U,$$
	\noindent where 
	$$\tau_U = \rank\nolimits_T(\varphi^-)(2g-2) + \lambda(B^*(\gamma,\gamma)B(\zeta, \zeta) - \rank\nolimits_T(\varphi^-)),$$
	\noindent also holds.
\end{theorem}

\noindent Proof. Let $e \in \mathfrak g_1$ be an element in the orbit where $\varphi^+$ takes values generically (see Remark \ref{genericorbit}) and consider the associated $\mathfrak{sl}_2$-triple $(h,e,f)$ with $h \in \mathfrak g_0$. Let $(\hat{G}_0,\hat{\mathfrak g}_1)$ be the maximal JM-regular prehomogeneous vector subspace for $e$ from Definition \ref{maxphv}. Let $s := \zeta - \frac{h}{2}$ and consider the associated parabolic Lie algebra $\mathfrak g_{0,s} \subseteq \mathfrak g_0$ of the parabolic subgroup $P_{0,s} \le G_0$. We have a reduction $\sigma \in H^0(E(G_0/P_{0,s}))$: first, it is well defined over the $x \in X$ where $\varphi^+(x)$ is in the orbit of $e$ (as $P_{0,s}$ is constructed from this value $e$), and since this happens generically, the section extends to $X$. Notice that the Toledo character can be split
$$\chi_T(x) = B^*(\gamma,\gamma)(B(\frac{h}{2},x) + B(s,x)),$$
\noindent so that, multiplying by an appropriate $q \in \mathbb \Z_{>0}$, lifting to $\tilde{\chi}_T$ on the group, applying to the reduction $E_\sigma$, taking degree and dividing by $q$ (the usual steps from previous definitions) we get:
\begin{equation}\label{am1}
	\tau(E,\varphi) = \deg E(\sigma, B^*(\gamma,\gamma)s) + \deg_{\hat{\chi}_T}(E_\sigma),
\end{equation}
\noindent where the rightmost term is the one corresponding to the character $\hat{\chi}_T : x \mapsto B^*(\gamma,\gamma)B(\frac{h}{2},x)$ of $\liep_{0,s}$.

Now we will get bounds for both of these terms. The Levi factor $L_{0,s}$ of $P_{0,s}$ is the centralizer of $s$ and thus it is $\hat{G}_0$. This means that the $P_{0,s}$-bundle $E_\sigma$ can be further reduced by projecting to the Levi factor, obtaining $E_\sigma(\hat{G}_0)$ a $\hat{G}_0$-bundle. Using the $\hat{G}_0$ action on $\hat{\lieg}_1$ we may construct the associated bundle $E_\sigma(\hat G_0)(\hat{\lieg}_1)$. By construction (recall that $\varphi^+$ takes values generically in the open $\hat{G}_0$-orbit of $\hat{\lieg}_1$) we have a section $\varphi^+ \in H^0(E_\sigma(\hat{G}_0)(\hat{\mathfrak g}_1) \otimes K_X)$. Let $F : \hat{\mathfrak g}_1 \to \mathbb C$ be the relative invariant of Lemma \ref{relinv} corresponding to the lift $\hat \chi_{T,q}$ of the character $q\hat{\chi}_T$. Notice that its degree is $q\rank_T(\hat{G}_0, \hat{\mathfrak g}_1) = q\hat{\chi}_T(\frac{h}{2})$. The relative invariant does not vanish at the open orbit $\hat{\Omega}$ (or else $F \equiv 0$), and $\varphi^+$ is generically in that orbit, so $F(\varphi^+)$ is a nonzero section of the line bundle $E_\sigma(\hat \chi_{T,q}) \otimes K_X^{q\hat{\chi}_T(\frac{h}{2})}$. This means that the degree of this line bundle is non-negative, namely
\begin{equation}\label{am2}
	\deg_{\hat{\chi}_T}(E_\sigma) = \frac{1}{q}\deg(E_\sigma(\hat{\chi}_{T,q})) \ge -\hat{\chi}_T\left(\frac{h}{2}\right)(2g-2),
\end{equation}
\noindent yielding a bound for one of the terms.

For the other term we use the $\alpha$-semistability. On one hand, $[s,e] = 0$ by definition, so $\varphi^+ \in H^0(E_\sigma(\mathfrak g_{1,s}) \otimes K_X)$. On the other hand, we also have $\mathfrak g_{1-m} \subseteq \mathfrak g_{1-m,s}$. In order to see this, let $x \in \mathfrak g_{1-m}$ and decompose in $\ad(\frac{h}{2})$-eigenvectors $x = \sum_{j}x_j$ with $x_j \in \mathfrak g_{1-m}$ corresponding to the eigenvalue $j$. This is possible due to the fact that $\mathfrak g_{1-m}$ is $\ad(\frac{h}{2})$-invariant. As $e \in \lieg_1$, we have $\ad(e)^{2m-1} \equiv 0$ and hence $j \in \{1-m, \frac{1-2m}{2}, \dots, \frac{2m-1}{2}, m-1\}$. Thus,
$$[s,x_j] = [\zeta - \frac{h}{2},x_j] = (1-m-j)x_j.$$ 
Since $1-m \le j \le m-1$, we have that $\lambda_j := (1-m-j) \le 0$. Thus $\text{Ad}(e^{ts})(x_j) = \exp(\ad(ts))(x_j) = e^{t\lambda_j}x_j$ is bounded as $t \to \infty$, and hence $x \in \mathfrak g_{1-m,s}$. In other words, we have seen that $\varphi^- \in H^0(E_\sigma(\mathfrak g_{1-m,s}) \otimes K_X)$ and thus $\varphi \in H^0(E_\sigma((\lieg_1 \oplus \lieg_{1-m})_s) \otimes K_X)$. We can then apply the definition of $\alpha$-semistability to get
\begin{equation}\label{am3}
	\deg E(\sigma, B^*(\gamma,\gamma)s) \ge B(\lambda\zeta, B^*(\gamma,\gamma)s).
\end{equation}
Combining equations (\ref{am1}), (\ref{am2}) and (\ref{am3}) results in
$$\tau(E,\varphi) \ge -\hat{\chi}_T\left(\frac{h}{2}\right)(2g-2) + \lambda B^*(\gamma,\gamma)B(\zeta, s).$$

What remains is just to rewrite $-\hat{\chi}_T(\frac{h}{2}) = B^*(\gamma,\gamma)B(\frac{h}{2},\frac{h}{2}) = B^*(\gamma,\gamma)(B(\zeta,\frac{h}{2}) + B(s,\frac{h}{2})) = B^*(\gamma,\gamma)(B(\zeta, \frac{h}{2})+0) = \rank_T(\varphi^+)$, as well as
$$\lambda B^*(\gamma,\gamma)B(\zeta, s) = \lambda(B^*(\gamma,\gamma)B(\zeta,\zeta)-B^*(\gamma,\gamma)B(\zeta, \frac{h}{2})) =$$$$= \lambda(B^*(\gamma,\gamma)B(\zeta, \zeta)-\rank\nolimits_T(\varphi^+)).$$

The upper bound is proven by replicating the argument with the prehomogeneous vector space $(G_0,\mathfrak g_{1-m})$ instead of $(G_0,\mathfrak g_1)$. This yields a lower bound for the Toledo invariant $\tau'$ of Remark \ref{dual}, which is related to $\tau$ by a negative constant and hence provides an upper bound for $\tau$. In order for the argument to work with $(G_0,\mathfrak g_{1-m})$, it is essential that the $\mathbb Z$-grading corresponding to this prehomogeneous vector space, which is the one given by $\frac{1}{1-m}\zeta$, also induces the pair $(G_0,\lieg_1 \oplus \lieg_{1-m})$. This happens when $m=2$, in which case we also have that $\tau' = -\tau$. If this is not the case, the only part of the argument that no longer works is the one using semistability, as it could be the case that $\varphi^+ \notin H^0(E_\sigma(\mathfrak g_{1,s}) \otimes K_X)$. The $\varphi^- = 0$ case is established in \cite[Theorem 5.3]{biquard_arakelov-milnor_2021}.\qed

The bound applies for $\alpha$-semistable pairs, so that in particular it holds in the moduli spaces $\mathcal M^\alpha(G_0,\lieg_1 \oplus \lieg_{1-m})$ of polystable elements. In the case of $\alpha = 0$, we get that for any $(G_0,\lieg_1 \oplus \lieg_{1-m})$-semistable Higgs pair $(E,\varphi)$ (in particular, for elements of $\mathcal M(G_0,\lieg_1 \oplus \lieg_{1-m})$) the following inequality is satisfied:
\[-(2g-2)\rank\nolimits_T(\varphi^+) \le \tau(E,\varphi).\]
By using the fact that $\rank_T(\varphi^+) \le \rank_T(G_0,\lieg_1)$, a (coarser) bound independent of the specific element is achieved:
$$-(2g-2)\rank\nolimits_T(G_0,\lieg_1) \le \tau(E,\varphi).$$

\begin{remark}
	For moduli spaces of $G^\R$-Higgs bundles, where $G^\R \subseteq G$ is a real form of Hermitian type (see Example \ref{realhermitian}), Theorem \ref{amw} also provides an upper bound. For $\alpha = 0$ this results in
	\[\tau(E,\varphi) \le (2g-2)\rank\nolimits_T(\varphi^-),\]
	\noindent as well as combined coarser bound independent of the specific element
	$$|\tau(E,\varphi)| \le (2g-2)\rank(G_0,\mathfrak g_1).$$
	
	This is the result presented in \cite{biquard_higgs_2017} generalizing previous specializations such as \cite{bradlow_surface_2003} for $G^\R = \SU(p,q)$. 
\end{remark}

\begin{remark}
	For the case associated to the $\Z$-gradings of $\liesl_n\C$ in Example \ref{linear_quiver}, expressing the Toledo invariant in terms of ranks and degrees of vector bundles (see Example \ref{exampletoledocyclic}) and the value $\rank_T(G_0, \lieg_1)$ in terms of the $d_i$ and $m$ makes it possible to derive the bound using the semistability via slope inequalities of invariant subbundles constructed from the kernels and images of the maps $E_i \to E_{i+1} \otimes K_X$ given by the Higgs field. This strategy was followed for the $m=2$ case in \cite{bradlow_surface_2003}. For $m > 2$, while still possible, results in cumbersome computations.
\end{remark}
\section{Cayley correspondence and maximal Toledo invariant}\label{cayley}
\subsection{Uniformising Higgs bundles}
In this section we are interested in $(G_0,\lieg_1 \oplus \lieg_{1-m})$-Higgs pairs whose Toledo invariant is \textit{maximal}, that is, it attains the bound from Theorem \ref{amw}. We denote by $\mathcal M^{\text{max}}(G_0,\lieg_1 \oplus \lieg_{1-m}) \subseteq \mathcal M(G_0, \lieg_1 \oplus \lieg_{1-m})$ the locus of polystable pairs with maximal Toledo invariant. We work in the case where $(G_0, \lieg_1)$ is JM-regular (this is the natural generalisation of symmetric pairs of \textit{tube type} \cite[Example 2.12]{biquard_arakelov-milnor_2021}, where a description of the pairs with maximal invariant exists \cite{biquard_higgs_2017}). Such pairs can be characterised as follows:
\begin{proposition}\label{maximaltoledo}
	Suppose that $(G_0,\mathfrak g_1)$ is JM-regular. A polystable $(G_0,\lieg_1 \oplus \lieg_{1-m})$-Higgs pair $(E,\varphi)$ is maximal if and only if $\varphi^+(x) \in \Omega \subseteq \mathfrak g_1$ for all $x \in X$, where $\Omega$ is the open orbit.
\end{proposition}
\noindent Proof. First, if $(E,\varphi)$ is maximal we need that $\rank(\varphi^+) = \rank_T(G_0,\mathfrak g_1)$ as we have $\tau(E,\varphi) \ge -\rank_T(\varphi^+)(2g-2)$. This implies that $\varphi^+$ is in $\Omega$ generically. With the notation from the proof of Theorem \ref{amw}, by JM-regularity we have $s=0$ and thus the only thing that we have to inspect is whether $\deg_{\hat{\chi}_T}(E_\sigma) = -\hat{\chi}_T(\frac{h}{2})(2g-2)$ holds. This happens if and only if the degree of the line bundle $E_\sigma(\hat{\chi}_{T,q}) \otimes K_X^{q\hat{\chi}_T(\frac{h}{2})}$ is zero, which happens if and only if its nonzero section $F(\varphi^+)$ is nonvanishing, meaning that $\varphi^+(x) \in \Omega$ for all $x \in X$. \qed

The space $\mathcal M^{\text{max}}(G_0, \lieg_1 \oplus \lieg_{1-m})$ is nonempty because already in the case $\varphi^- \equiv 0$ there are polystable $(G_0,\mathfrak{g}_1)$-Higgs pairs attaining this bound \cite[Proposition 6.1]{biquard_arakelov-milnor_2021} (we also explicitly showcase some pairs with $\varphi^- \not\equiv 0$ in Example \ref{maximalexample}). We recall the construction. 

Let $e \in \Omega \subseteq \mathfrak g_1$ be an element of the open orbit and complete it to an $\mathfrak {sl}_2$-triple $(h,e,f)$ with $h \in \mathfrak g_0$, $f \in \mathfrak g_{-1}$. Let $S \le G_0$ be the connected subgroup with Lie algebra $\gen{e,f,g} \subseteq \mathfrak g$, which can be isomorphic to $\text{PSL}_2(\mathbb C)$ or $\SL_2(\mathbb C)$ depending on $G$ and the triple. Let $C \le G$ be the reductive subgroup centralizing $\{h,e,f\}$, which coincides with $G_0^e \le G_0$: it has to be contained in $G_0$ in order to stabilize $h$ and it should also stabilize $e$, and this is sufficient. Let $T \le S$ be the connected subgroup with Lie algebra $\gen{h}$. We have $T \le G_0$ as $h \in \mathfrak g_0$. There are two cases:

\begin{itemize}
	\item If $S \simeq \SL_2(\mathbb C)$, it is simply connected and thus the representation $\gen{h} \to \mathfrak{gl}(\gen{e})$ given by $\lambda h \mapsto \ad(\lambda h)$ lifts to a representation $\mathbb C^* \simeq T \to \GL(\gen{e})$, this being just the adjoint representation. As $[h,e] = 2e$, the $\mathbb C^*$-action we get on $\gen{e}$ via this lift is $\lambda \cdot e = \lambda^2e$. Choose a square root $K_X^{\frac{1}{2}}$ (this can be done as $\deg K_X = 2g-2$ is even) and let $E_T$ be the frame bundle for $K_X^{-\frac{1}{2}}$, in other words, the $\mathbb C^*$-bundle such that the bundle associated to the standard representation of $\mathbb C^*$ in $\mathbb C$ is $E_T(\mathbb C) \simeq K_X^{-\frac{1}{2}}$. Using the isomorphism $\mathbb C^* \simeq T$ we have a $T$-bundle, and since $T$ acts on $\gen{e}$ with weight $2$ we have $E_T(\gen{e}) \simeq (K_X^{-\frac{1}{2}})^2 = K_X^{-1}$. This means that $E_T(\gen{e}) \otimes K_X \simeq \mathcal O$ so we can define a constant section $e \in H^0(X, E_T(\gen{e}) \otimes K_X)$. Extending the structure group gives $(E_T(G_0),e)$ a $(G_0,\gen{e})$-Higgs pair which is in particular a $(G_0,\lieg_1 \oplus \lieg_{1-m})$-Higgs pair.
	
	\item If $S \simeq \text{PSL}_2(\mathbb C)$, we can take its universal cover $\SL_2(\mathbb C) \to S$ which is of degree two. The torus $T \subseteq S$ lifts to $\hat{T} \subseteq \SL_2(\mathbb C)$. We have that $\hat{T}$ is a double cover of $T$ and there are isomorphisms with $\mathbb C^*$ such that map $\hat{T} \simeq \mathbb C^* \to T \simeq \mathbb C^*$ is given by $\lambda \mapsto \lambda^2$. By the previous argument the adjoint action of $\hat{T} \simeq \mathbb C^*$ on $\gen{e}$ is given by $\lambda \cdot e = \lambda^2 e$, so that it descends to $T$ as $\lambda \cdot e = \lambda e$. Now we let $E_T$ be the frame bundle of $K_X^{-1}$. The associated bundle is $E_T(\gen{e}) \simeq K_X^{-1}$ and hence $E_T(\gen{e}) \otimes K_X \simeq \mathcal O$. Thus $e$ defines a holomorphic section of $E_T(\gen{e}) \otimes K_X$. Extending the structure group gives $(E_T(G_0),e)$ a $(G_0,\gen{e})$-Higgs pair and in particular a $(G_0,\lieg_1 \oplus \lieg_{1-m})$-Higgs pair.
\end{itemize} 

The resulting $(E_T(G_0),e)$ is called \textbf{uniformising Higgs bundle} and is polystable (see \cite{hitchin_self-duality_1987}). By construction $\varphi^+ = e$ is always in the open orbit, hence its Toledo invariant attains the bound.
\subsection{Cayley correspondence}
Now we extend the description of $\mathcal M^{\text{max}}(G_0, \lieg_1 \oplus \lieg_{1-m})$ from \cite{biquard_arakelov-milnor_2021} (in the case $\varphi^- \equiv 0$) and \cite{biquard_higgs_2017} (in the case $m=2$). 

Consider the $\mathfrak{sl}_2\mathbb C$-representation on $\mathfrak g$ given by the triple $(h,e,f)$. Let $W \subseteq \mathfrak g$ be the direct sum of all the subspaces corresponding to the irreducible subrepresentations of dimension $2m-1$. Notice that this is the maximum possible such dimension, since $e \in \mathfrak g_1$ implies $\ad(e)^{2m-1} \equiv 0$. Define the subspace $V := \mathfrak g_0 \cap W \subseteq \mathfrak g_0$.

\begin{proposition}\label{cayleyisomorphism}
	Suppose that the prehomogeneous vector space $(G_0,{\mathfrak g}_1)$ is JM-regular. The map $\ad(e)^{m-1} : \mathfrak g_{1-m} \to V$ is an isomorphism.
\end{proposition}

\noindent Proof. By the structure of irreducible representations of $\mathfrak{sl}_2\mathbb C$ and the fact that the maximum possible dimension of an irreducible representation is $2m-1$, the eigenvalue $2(1-m)$ for $\ad(h)$ can only appear on elements of $W$ and hence we have the inclusion $\ker\left(\ad(h)-2(1-m)\Id\right) \subseteq W$. By JM-regularity, $h=2\zeta$ and thus this space is precisely $\mathfrak g_{1-m}$, i.e. $\mathfrak g_{1-m} \subseteq W$. As $W$ is a subrepresentation, we have for all $j$ that $\ad(e)^j(W) \subseteq W$. Moreover, since $e \in \mathfrak g_1$, we have $\ad(e)^{m-1}(\mathfrak g_{1-m}) \subseteq \mathfrak g_0$. We conclude that $\ad(e)^{m-1}(\mathfrak g_{1-m}) \subseteq V$ so the map is well defined.

It is an isomorphism: suppose that in $W$ there are $n$ summands of the irreducible $\mathfrak{sl}_2\mathbb C$-representation of dimension $2m-1$. Let $B := \{v_1,\dots,v_n\}$ be $n$ linearly independent eigenvectors for the eigenvalue $2(1-m)$ of $\ad(h)$, one in each irreducible representation. We know from the previous argument that $\mathfrak g_{1-m}= \gen{B}$. Now, $B' := \{\ad(e)^{m-1}(v_1),\dots, \ad(e)^{m-1}(v_n)\}$ are linearly independent (because each of them is in a different irreducible representation) and by the structure of irreducible representations we have $\gen{B'} = V$. \qed

Recall the subgroup $C = G_0^e \le G_0$. The adjoint representation of $C$ restricts to an action on $V$: if $\mathfrak c \subseteq \mathfrak g_0$ is the Lie algebra of $C$, and we take $c \in \mathfrak c$ and $\ad^{m-1}(e)(x) \in V$ (where $x \in \mathfrak g_{1-m}$), we have that $[c,\ad^{m-1}(e)(x)] = \ad^{m-1}(e)([c,x]) \in V$, where we used for the last step that $[c,e] = 0$ and that $[c,x] \in \mathfrak g_{1-m}$. Thus, we can consider $(C,V)$-Higgs pairs.

We also need the following construction: suppose that $E_C$ is a principal $C$-bundle and recall the uniformising $T$-bundle $E_T$ from before. Since $m : T \times C \to G_0$ is a group homomorphism (because $T$ commutes with $C$, this follows from the fact that if $x \in \mathfrak g_0$ centralizes $e$, then $[e, [h,x]] = [[e,h],x] + [h,[e,x]] = [-2e,x]+0 = 0$), it is possible to define the $G_0$-bundle 
$$(E_T\otimes E_C)(G_0) := (E_T \times E_C) \times_m G_0.$$
This is a notion of \textit{tensor product} for principal bundles, and it works similarly for any two commuting subgroups of $G_0$ (for example, we can tensor any $G_0$-bundle by any bundle for a central subgroup such as $T$). As with vector bundles, given metrics $h_T$ and $h_C$ on each bundle respectively, there is a well defined product metric $h_T \otimes h_C$ on $E_T\otimes E_C(G_0)$ and the curvatures of the Chern connections verify $F_{h_T\otimes h_C} = F_{h_T} + F_{h_C}$.

We can now state and prove the main result of this section, extending \cite[Theorem 6.2]{biquard_arakelov-milnor_2021} and \cite[Theorem 5.2]{biquard_higgs_2017}.

\begin{theorem}[Cayley correspondence]\label{caycorr}
	Suppose that the prehomogeneous vector space $(G_0,\mathfrak g_1)$ is JM-regular. There is an injective morphism of complex algebraic varieties
	$$\mathcal M^\text{\emph{max}}(G_0, \lieg_1 \oplus \lieg_{1-m}) \to \mathcal M_{K_X^m}(C,V).$$
\end{theorem}
\noindent Proof. First we will define the map, which is a version of the \textit{global Slodowy slice map} from \cite{collier_gp-opers_2021} which, in turn, generalises the map behind the Cayley correspondence for Hermitian forms \cite{biquard_higgs_2017} and magical $\mathfrak{sl}_2$-triples \cite{bradlow_general_2021}. The definition uses the notation and elements introduced throughout the section, and is an extension of the construction given in \cite{biquard_arakelov-milnor_2021} for the case $\varphi^- = 0$. Take $(E,\varphi) \in \mathcal M^\text{max}(G_0, \lieg_1 \oplus \lieg_{1-m})$. We will use the fact that $E(\mathfrak g_1) \otimes K_X \simeq (E_T^{-1}\otimes E)(\mathfrak g_1)$, which is true in both cases $S \simeq \SL_2(\mathbb C)$ or $S \simeq \text{PSL}_2(\mathbb C)$ as can be seen, for example, checking that the transition functions for both agree.

Since the Toledo invariant is maximal, by Proposition \ref{maximaltoledo} we have $\varphi^+(x) \in \Omega$ for all $x \in X$. Thus, $\varphi^+ \in H^0(X, E(\Omega) \otimes K_X) = H^0(X, (E_T^{-1} \otimes E)(\Omega))$. As $C$ is the stabilizer of the element $e \in \Omega$ in $G_0$, we have $\Omega = G_0/C$. In other words, $\varphi^+ \in H^0((E_T^{-1} \otimes E)(G_0/C))$ is a reduction of the structure group of $E_T^{-1} \otimes E$ from $G_0$ to $C$. Let $E_C$ be the reduced $C$-bundle. This is the first part of the desired $(C,V)$-Higgs pair. Notice that we have $E_C(G_0) \simeq E_T^{-1} \otimes E$ which in turn gives the relation
$$E \simeq E_T \otimes E_C(G_0) = (E_T \otimes E_C)(G_0).$$ 
For the second part, we have as before that $(E_T \otimes E_C)(\gen{e}) \otimes K_X = E_C(\gen{e}) \otimes K_X^{-1} \otimes K_X = E_C(\gen{e})$. Since $C$ centralises $e$, there is a well-defined constant section $e \in H^0(X,E_C(\gen{e})) = H^0(X,(E_T \otimes E_C)(\gen{e}) \otimes K_X) \subseteq H^0(X,(E_T \otimes E_C)(\mathfrak g_1) \otimes K_X)$. Using this section together with Proposition \ref{cayleyisomorphism} gives a vector bundle isomorphism:
$$\ad(e)^{m-1} : (E_T \otimes E_C)(\mathfrak g_{1-m}) \otimes K_X \to (E_T \otimes E_C)(V) \otimes K_X^m.$$
Since $T$ acts trivially on $\mathfrak g_0$ (because $[h,\mathfrak g_0] = 0$ by JM-regularity) we get that $(E_T \otimes E_C)(V) \simeq E_C(V)$. We can now take $\varphi^{-} \in H^0(X, E(\mathfrak g_{1-m}) \otimes K_X) = H^0(X, (E_T \otimes E_C)(G_0)(\mathfrak g_{1-m}) \otimes K_X) = H^0(X, (E_T \otimes E_C)(\mathfrak g_{1-m}) \otimes K_X)$ and apply the previous to get 
$$\varphi' := \ad(e)^{m-1}(\varphi^-) \in H^0(X,E_C(V) \otimes K_X^m).$$ 
The Cayley correspondence map is then
$$(E,\varphi) \mapsto (E_C,\varphi').$$ 
The inverse has already been hinted at throughout the proof, but we collect it now. Given $(E_C,\varphi')$ a $(C,V)$-Higgs pair, we set 
$$E := (E_T \otimes E_C)(G_0),$$ 
$$\varphi^+ := e \in H^0(X,E(\mathfrak g_1) \otimes K_X),$$ 
$$\varphi^- := (\ad(e)^{m-1})^{-1}(\varphi') \in H^0(X, E(\mathfrak g_{1-m}) \otimes K_X).$$ 
Its Toledo invariant is the desired one (we know this if we start with $(E_C,\varphi')$ which is the image of some $(E,\varphi)$ with maximal invariant, but we shall see that it holds no matter the starting $(C,V)$-pair $(E_C,\varphi')$). We cannot simply argue that $\varphi^+ \in \Omega$ and use Proposition \ref{maximaltoledo} because we do not yet know whether $(E,\varphi^++\varphi^-)$ is polystable. However, we can compute the invariant. First notice that if $\mathfrak c \subseteq \mathfrak g_0$ is the Lie algebra for $C$, the Toledo character vanishes: $\chi_T(\mathfrak c) \equiv 0$. This is because if $c \in \mathfrak c$, we have $2B(\zeta,c) = B(h,c) = B([e,f],c) = -B(f,[e,c]) = -B(f,0) = 0$. Thus, any lift $\tilde{\chi}_T : G_0 \to \mathbb C^*$ of some $q\chi_T$ satisfies $\tilde{\chi}_T|_C \equiv 1$, meaning that the bundles $(E_T \otimes E_C)(G_0) \times_{\tilde{\chi}_T} \mathbb C^*$ and $E_T(G_0)\times_{\tilde{\chi}_T} \mathbb C^*$ are the same (the transition functions agree). Thus
$$\tau(E,\varphi) = \tau(E_T(G_0),e),$$
\noindent and $(E_T(G_0),e)$ is a uniformising Higgs bundle, which is maximal.

So far we have established a correspondence between $(G_0,\lieg_1 \oplus \lieg_{1-m})$-Higgs pairs with maximal Toledo invariant (equal to $-(2g-2)\rank_T(G_0, \mathfrak g_1)$) and $K_X^m$-twisted $(C,V)$-Higgs pairs. Now we see that it restricts to the moduli space, in other words, that if we start with a polystable $(E,\varphi) \in \mathcal M^{\text{max}}(G_0, \lieg_1 \oplus \lieg_{1-m})$ the resulting $(E_C,\varphi')$ is polystable. For this, pick a maximal compact $C^\R \subseteq C$ with Lie algebra $\mathfrak c^\R \subseteq \mathfrak c$. Let $s \in i\mathfrak c^{\R,\Gamma}$, $P'_s \subseteq C$ the associated subgroup and $\sigma' \in H^0(X, E_C(C/P'_s))$ a reduction of structure group such that $\varphi' \in H^0(X, E_{C,\sigma'}(V_s) \otimes K_X^m)$. This element $s$ regarded in $i\mathfrak k^{\Gamma}$ for $\mathfrak k$ the Lie algebra of a maximal compact $K$ in $G_0$ containing $C^\R$ also defines a subgroup $P_s \subseteq G_0$, which by definition verifies $P_s' \subseteq P_s$. Thus we have a map $C/P_s' \to G_0/P_s$ which, from $\sigma'$, gives a reduction $\sigma \in H^0(X,E(G_0/P_s))$. Now we need to verify that $\varphi \in H^0(X, E_\sigma((\lieg_1 \oplus \lieg_{1-m})_{1,s}) \otimes K_X)$, in other words, that $e \in H^0(X, E_\sigma(\mathfrak g_{1,s}) \otimes K_X)$ and $(\ad(e)^{m-1})^{-1}(\varphi') \in H^0(X, E_\sigma(\mathfrak g_{1-m,s}) \otimes K_X)$. The former is due to the fact that $s \in i\mathfrak c^\R \subseteq \mathfrak c$, so that $[s,e] = 0$. The latter follows because (pointwise) $\ad(e)^{m-1}$ restricts to an isomorphism between $\mathfrak g_{1-m,s}$ and $V_s$, since for $x \in \mathfrak g_{1-m}$ we have $\ad(e)^{m-1}([s,x]) = [s,\ad(e)^{m-1}(x)]$ as $[s,e] = 0$. Thus, since $\varphi'$ takes values in $V_s$, we have that $\varphi^-$ takes values in $\mathfrak g_{1-m,s}$ as desired. Polystability of $(E,\varphi)$ then gives
$$\deg E(\sigma,s) \ge 0.$$
Now, recall from before that $\mathfrak c$ and $\gen{h}$ are orthogonal via $B$. In particular, $B(s,h) = 0$. This means that the bundles $E_\sigma \times_{\tilde{\chi}_s} \mathbb C^* = ((E_T \otimes E_C)(G_0))_\sigma \times_{\tilde{\chi}_s} \mathbb C^* $ and $E_{C,\sigma'} \times_{\tilde{\chi}_s} \mathbb C^*$ are the same (or, using the differential geometric definition of the degree, that $\chi_s(F_{h_T} + F_{h_C}) = \chi_s(F_{h_C})$). Thus
$$\deg E_C(\sigma',s) = \deg E(\sigma,s) \ge 0.$$
This shows semistability, and polystability follows after checking that if $\deg E_C(\sigma',s) = 0$, by the equality of the degrees above and polystability of $(E,\varphi)$ we get a reduction $\sigma''$ of $E$ to the subgroup $L_s \subseteq P_s$. But then $E_{C}' := (E_{\sigma''} \otimes E_T^{-1}) \cap E_C$ is a reduction of $E_C$ to $L'_s \subseteq P'_s$.

Thus the map restricts from $\mathcal M^{\text{max}}(G_0,\lieg_1 \oplus \lieg_{1-m})$ to $\mathcal M_{K_X^m}(C,V)$ and it is injective due to the inverse shown above. \qed

As explained in the proof, the map between $(G_0,\lieg_1 \oplus \lieg_{1-m})$-Higgs pairs and $K_X^m$-twisted $(C,V)$-Higgs pairs is a bijection, and its restriction to the moduli spaces is an injection. A proof of surjectivity in the moduli spaces exists via the Hitchin--Kobayashi correspondence of Proposition \ref{hitkob}, for which we must add the extra assumption that a similar correspondence (i.e. with the same equation as in Proposition \ref{hitkob}) exists for $(C,V)$-Higgs pairs. This is the case if $(C,V) = (H_0, \lieh_1 \oplus \lieh_{1-m})$ for a $\Z$-grading of some semisimple Lie group $H$ with Lie algebra $\lieh$ or, more generally, if $(C,V)$ is a Vinberg $\theta'$-pair for some $\theta' \in \Aut_{m'}(H)$.

\begin{proposition}\label{cayleysurj}
	Suppose that the prehomogeneous vector space $(G_0,\mathfrak g_1)$ is JM-regular and that $(C,V)$ is a Vinberg $\theta'$-pair for some semisimple Lie group $H$ and $m'$ with $\theta \in \Aut_{m'}(H)$. Then, the map from Theorem \ref{caycorr} is an isomorphism of complex algebraic varieties.
\end{proposition}

\noindent Proof. The only thing that remains to be shown is that the inverse descends to the moduli spaces, meaning that it sends polystable bundles $(E_C,\varphi')$ to polystable bundles $(E,\varphi)$. This uses the Hitchin--Kobayashi correspondence from Proposition \ref{hitkob}. Let $\tau: \mathfrak g \to \mathfrak g$ be an involution defining a compact real form $K \subseteq G$ with $\tau(\lieg_i \oplus \lieg_{i-m}) = \lieg_{-i} \oplus \lieg_{m-i}$ for $i \neq 0$, $\tau(\lieg_0) = \lieg_0$ and chosen so that $\tau(e) = -f$.

First, by polystability of the uniformising Higgs bundle $(E_T,e)$ and the Hitchin--Kobayashi correspondence, there exists a metric $h_T$ on $E_T$ such that its curvature $F_{h_T}$ satisfies
$$F_{h_T} + [e,-\tau_{h_T}(e)]\omega = 0$$
\noindent which means that $F_{h_T} = -[e,-f]\omega = h\omega$, that is, the curvature is constant.

On the other hand, since $(E_C, \varphi')$ is polystable, there is a metric $h_C$ on $E_C$ such that, after fixing a metric on $K_X$ and using it to define a metric on the power $L=K_X^m$ which defines an involution $\tau_{h_C}$, we have
$$F_{h_C} + [\varphi', -\tau_{h_C}(\varphi')]\omega = 0,$$
\noindent for $\omega$ a Kähler form on $X$.

Using these we can take on $E = (E_T \otimes E_C)(G_0)$ the product metric $h_T \otimes h_C$, and its curvature is $F_{h_T} + F_{h_C}$ which satisfies:
$$F_{h_T} + F_{h_C} = h\omega - [\varphi', -\tau_{h_C}(\varphi')]\omega,$$
\noindent which can be rewritten, using the fact that $\varphi' = \ad^{m-1}(e)(\varphi^-)$ and that $\tau_{h_C} = \tau_{h_T \otimes h_C}$, as
$$F_{h_T} + F_{h_C} + [\ad\nolimits^{m-1}(e)(\varphi^-), -\tau_{h_T \otimes h_C}(\ad\nolimits^{m-1}(e)(\varphi^-))] = h\omega = -i(ih)\omega.$$
Notice that $ih$ is central in $\mathfrak g_0$ (by JM-regularity) and it belongs in $\mathfrak k_0$ because $\tau(ih) = -i\tau([e,f]) = -i[-f,-e] = i[e,f] = ih$. Thus we can claim by the Hitchin--Kobayashi correspondence that the $G_0$-Higgs bundle $(E, \ad^{m-1}(e)(\varphi^-))$ is $\alpha$-polystable for $\alpha := ih \in \mathfrak z_0$.

Now we show the polystability of $(E,\varphi^++\varphi^-)$. Let $s \in i\mathfrak k_0$ and consider a reduction $\sigma \in H^0(X,E(G_0/P_s))$ with $\varphi = \varphi^++\varphi^- \in H^0(X, E_\sigma(\mathfrak g_{1,s} \oplus \mathfrak g_{1-m,s}) \otimes K_X)$. This means that $\varphi^+ = e \in H^0(X, E_\sigma(\mathfrak g_{1,s}) \otimes K_X)$ and $\varphi^- \in H^0(X, E_\sigma(\mathfrak g_{1-m,s}) \otimes K_X)$, so that $\ad^{m-1}(e)(\varphi^-) \in H^0(X, E_\sigma(V_s) \otimes K_X^m)$. The $ih$-polystability of $(E, \ad^{m-1}(e)(\varphi^-))$ then implies
$$\deg E(\sigma, s) \ge B(ih,s).$$
We show in Lemma \ref{polyineq} below that $B(ih,s) \ge 0$, so $(E,\varphi)$ is semistable. For polystability, if $\deg E(\sigma, s) = 0$ then $B(ih,s) = 0$ and thus $ih$-polystability of $(E, \ad^{m-1}(e)(\varphi^-))$ gives a reduction $\sigma'$ to a Levi $L_s \subseteq P_s$ such that $\ad^{m-1}(e)(\varphi^-) \in H^0(X, E_{\sigma'}(V_s^0) \otimes K_X^m)$. Using the second part of Lemma \ref{polyineq} we have that $e \in \mathfrak g_{1,s}^0$, so that $\ad^{m-1}(e)$ is an isomorphism between $\mathfrak g_{1-m,s}^0$ and $V_s^0$ and thus we get that $\varphi^- \in H^0(X, E_{\sigma'}(\mathfrak g_{1-m,s}^0) \otimes K_X)$ and hence $\varphi \in H^0(X, E_{\sigma'}((\lieg_1 \oplus \lieg_{1-m})_s^0) \otimes K_X)$, yielding polystability of $(E,\varphi)$ and completing the proof. \qed

We now establish the lemma used above.

\begin{lemma}\label{polyineq}
	Let $s \in i \mathfrak k_0$ and suppose that $e \in \mathfrak g_{1,s}$. Then, $B(ih,s) \ge 0$ with equality if and only if $e \in \mathfrak g_{1,s}^0$.
\end{lemma}

\noindent Proof. Consider the $\mathcal O$-twisted $(G_0,\mathfrak g_1)$-Higgs pair given by the trivial principal bundle $X \times G_0$ and the constant section $e \in H^0(X, (X \times G_0)(\mathfrak g_1)) = H^0(X, X \times \mathfrak g_1)$. As the principal bundle is trivial we can choose a metric $h$ with $F_h = 0$, as well as the constant metric on $L=\mathcal O$ so that
$$F_h + [e,\tau_h(e)]\omega = 0 + [e,-f]\omega = -h\omega = -i(-ih)\omega,$$
\noindent which by the Hitchin--Kobayashi correspondence of Proposition \ref{hitkob} means that this bundle is $\alpha$-polystable for $\alpha = -ih$. Because the bundle is trivial we can take a constant reduction $\sigma$ to $P_s$ and the fact that $e \in \mathfrak g_{1,s}$ yields by polystability that $\deg (X \times G_0)(\sigma,s) \ge B(-ih,s)$. Since $(X \times G_0)$ is trivial, $\deg (X \times G_0)(\sigma,s) = 0$ (this is immediately seen from either of the definitions of degree) so that $0 \ge B(-ih,s)$ as desired. Equality implies by polystability that $e \in \mathfrak g_{1,s}^0$, and if $e \in \mathfrak g_{1,s}^0$ then $B(ih,s) = B(i[e,f],s) = -B(if,[e,s]) = -B(if,0) = 0$. \qed 

\begin{example}\label{maximalexample}
	We now give an example of the Cayley correspondence for $m=3$. We work with the grading of Example \ref{linear_quiver} with ranks $(1,1,1)$. It is JM-regular as all the ranks are equal. In this case the isomorphism $\mathfrak g_{-2} = \mathbb C \to V$ is given by sending $\lambda$ to $\begin{pmatrix}
		\lambda & 0 & 0\\
		0 & -2\lambda & 0 \\
		0 & 0 & \lambda
	\end{pmatrix}$. The centralizer of $e = \begin{pmatrix}
		0 & 0 & 0 \\
		1 & 0 & 0\\
		0 & 1 & 0
	\end{pmatrix}$ is $C = \mu_3 = \left\{
	\lambda \cdot Id_3 : \lambda^3 = 1\right\}$ with Lie algebra $\mathfrak c = 0$. Because $V$ is abelian we have that $\mathfrak c \oplus V$ is a Cartan decomposition and $(C,V)$ a symmetric pair. Thus the Cayley correspondence (including the surjectivity) applies, and there is a bijection between $\mathcal M^{\text{max}}(G_0,\lieg_1 \oplus \lieg_{1-m})$ and $\mathcal M_{K_X^3}(C,V)$. In this case, since $C$ acts trivially on $V$, the Higgs field is $\omega \in H^0(X,K_X^3)$. The uniformising Higgs bundle is given by the vector bundle $E = K_X \oplus \mathcal O \oplus K_X^{-1}$ and the Higgs field
	$$e = \begin{pmatrix}
		0 & 0 & 0 \\
		1 & 0 & 0 \\
		0 & 1 & 0
	\end{pmatrix},$$
	\noindent so that, for example, the maximal bundles associated to the trivial $C$-bundle $X \times \mu_3 \to X$ are given by the vector bundle $E$ and Higgs field
	$$\varphi = \begin{pmatrix}
		0 & 0 & \omega \\
		1 & 0 & 0 \\
		0 & 1 & 0
	\end{pmatrix}.$$
\end{example}

\begin{remark}
	For the previously studied cases $\varphi^- \equiv 0$ \cite{biquard_arakelov-milnor_2021} and $m=2$ \cite{biquard_higgs_2017}, the correspondence of Theorem \ref{caycorr} is always bijective, as the pair $(C,V)$ in each case (of the form $(C,0)$ in the first case, and a symmetric pair in the second) satisfies the desired Hitchin--Kobayashi equation. However, this is no longer true in general, as showcased in Example \ref{example222} below.
\end{remark}

\begin{example}\label{example222}
	The case of ranks $(2,2,2)$ in the $\Z$-grading of Example \ref{linear_quiver} does not verify that $(C,V)$ is a Vinberg $\theta$-pair. In this case, 
	$$C = \left\{\begin{pmatrix}
		X & 0 & 0\\
		0 & X & 0\\
		0 & 0 & X
	\end{pmatrix} : X \in \SL\nolimits_2(\mathbb C)\right\},$$
	\noindent and
	$$V = \left\{\begin{pmatrix}
		X & 0 & 0\\
		0 & -2X & 0\\
		0 & 0 & X
	\end{pmatrix} : X \in \mathfrak{gl}_2\mathbb C\right\}.$$
	
	If we select the elements $v,v' \in V$ corresponding to $X=\begin{pmatrix}
		0 & -1 \\
		1 & 0
	\end{pmatrix}$ and $X' = \begin{pmatrix}
		0 & 1\\
		1 & 0
	\end{pmatrix}$ respectively, we get, for $B:=\begin{pmatrix}
		1 & 0 \\
		0 & -1
	\end{pmatrix}$, that
	$$[v,v'] = \begin{pmatrix}
		2B & 0 & 0\\
		0 & 8B & 0\\
		0 & 0 & 2B
	\end{pmatrix} = \begin{pmatrix}
		4B & 0 & 0\\
		0 & 4B & 0\\
		0 & 0 & 4B
	\end{pmatrix} + \begin{pmatrix}
		-2B & 0 & 0\\
		0 & 4B & 0\\
		0 & 0 & -2B
	\end{pmatrix},$$
	\noindent where the first summand is in $\mathfrak c$ and the second in $V$. The fact that $[V,V]$ has nonzero projection to both $V$ and $\mathfrak c$ means that $(C,V)$ is never a Vinberg $\theta$-pair. Thus, the Cayley correspondence injects the space of maximal $(G_0, \lieg_1 \oplus \lieg_{1-m})$-Higgs pairs into the space of $(C,V)$-bundles, but surjectivity is not guaranteed by the previous results.
\end{example}

\section{Application to the quaternionic $\mathbb Z$-grading}\label{quater}

\subsection{The quaternionic $\mathbb Z$-grading.}\label{quatgrad}

Let $\lieg$ be a complex simple Lie algebra. As explained in Example \ref{quaternionic}, $\lieg$ has a distinguished \textit{quaternionic} $\mathbb Z$-grading which is of the form $\lieg_{-2} \oplus \lieg_{-1} \oplus \lieg_0 \oplus \lieg_1 \oplus \lieg_2$, yielding an example for the case $m=3$. 

We recall the construction of this grading (see \cite{clerc_prehomogeneous_2003,wallach_quaternionic_1996}). Let $\liet \subseteq \lieg$ be a Cartan subalgebra and $\Delta$ the associated root system with a fixed choice of positive roots $\Delta^+ \subseteq \Delta$. Let $\beta \in \Delta^+$ be the highest root and $T_\beta \in \liet$ its coroot with respect to an invariant non-degenerate bilinear form $B^*$ on $\liet^*$, normalised so that $B^*(\beta, \beta) = 2$. Then, $\ad T_\beta$ acts on the root spaces as multiplication by a scalar in $\{-2,-1,0,1,2\}$, inducing the desired grading. We then have that the grading element is $\zeta := T_\beta$ and, denoting by $\lieg_\alpha$ the root space for $\alpha \in \Delta$, we have
\[\lieg_j = \begin{cases}
	\bigoplus_{\alpha \in \Delta_j}\lieg_\alpha, & \text{ if } j \neq 0\\
	\liet \oplus \bigoplus_{\alpha \in \Delta_0}\lieg_\alpha,& \text{ if } j = 0
\end{cases},\]
\noindent for $\Delta_j := \{\alpha \in \Delta : B^*(\alpha, \beta) = j\}$.

We have that $\lieg_2 = \lieg_\beta$, $\lieg_{-2} = \lieg_{-\beta}$ (hence both pieces are one-dimensional) and $\lieg_{-2} \oplus \C T_\beta \oplus \lieg_{2} \subseteq \lieg$ is a Lie subalgebra isomorphic to $\liesl_2\mathbb C$. This implies the following:

\begin{lemma}\label{minustwojmreg}
	The Vinberg $\C^*$-pairs $(G_0,\lieg_{2})$ and $(G_0,\lieg_{-2})$ are JM-regular.
\end{lemma}
\noindent Proof. Since $\lieg_{-2} \oplus \C T_\beta \oplus \lieg_{2}$ is isomorphic to $\liesl_2\mathbb C$, we can find an $\liesl_2$-triple $(T_\beta, e, f)$ with $e \in \lieg_2$, $f \in \lieg_{-2}$. Now, since $T_\beta$ is two times the grading element of $\lieg_{-2} \oplus \lieg_0 \oplus \lieg_2$, it suffices to see that $e$ and $f$ are in the respective open orbits, but this follows because $\lieg_{2}$ and $\lieg_{-2}$ are one-dimensional. \qed 

An alternative way to construct this grading \cite{bertram_quaternionic_2002} is from a quaternionic symmetric space $M=G^\R/H^\R$, which is a symmetric space whose holonomy group is contained in $\Sp(n)\Sp(1) \subseteq \SO(4n)$. In this case, there exist almost complex structures $I, J, K \in \Gamma(\End(TM))$, spanning at each point $p \in M$ a quaternionic subalgebra of $\End(T_pM)$ \cite[Section 14D]{besse2007einstein}. Now, if $\lieg = \lieh \oplus \liem$ is the decomposition from Example \ref{realhermitian}, there is a natural way \cite[Section 1.3]{bertram_quaternionic_2002} of viewing $I,J,K \in \lieh$, and the desired $\Z$-grading is then given by $\ad(I)$. From this point of view one has $\lieh = \lieg_{-2} \oplus \lieg_0 \oplus \lieg_2$ and $\liem = \lieg_{-1} \oplus \lieg_1$. 

There is a correspondence (see \cite[Section 3.1]{bertram_quaternionic_2002} and references therein) between compact quaternionic symmetric spaces and simple complex Lie groups, resulting in the fact that we get a $\Z$-grading of the desired form in every complex simple Lie algebra $\lieg$, which agrees (up to inner automorphism) with the one constructed above from the highest root.

\subsection{Arakelov--Milnor--Wood inequalities for the quaternionic grading.} 

Let $M=G^\R/H^\R$ be a quaternionic symmetric space, with $G^\R \subseteq G$ a real form of the complex simple Lie group $G$ with Lie algebra $\lieg$. Consider the quaternionic grading $\lieg = \bigoplus_{j=-2}^2\lieg_j$ explained above. The group $H$ has Lie algebra $\lieg_{-2} \oplus \lieg_0 \oplus \lieg_2$, and the existing theory for the Toledo invariant in Hermitian symmetric spaces \cite{biquard_higgs_2017} (that is, Theorem \ref{amw} for $m=2$), shows that the Toledo invariant for $G^\R$-Higgs bundles (i.e, for $(H,\lieg_{-1} \oplus \lieg_1)$-Higgs pairs) is bounded from above and from below.

Now let $G_0 \subseteq G$ be the connected subgroup with Lie algebra $\lieg_0$. We will apply the Arakelov--Milnor--Wood inequality of Theorem \ref{amw} to show that the Toledo invariant for $(G_0, \lieg_1 \oplus \lieg_{-2})$-Higgs pairs is bounded from below and, using the specific structure of the quaternionic grading, we will also get a new bound from above. These Higgs pairs are no longer related, from the point of view of non-abelian Hodge theory, to the quaternionic real form $G^\R$.

\begin{theorem}\label{amwquaternionic}
	Let $G$ be a complex simple Lie group and $\lieg = \bigoplus_{j=-2}^{2}\lieg_j$ the quaternionic $\Z$-grading on its Lie algebra. Let $G_0 \subseteq G$ be the connected subgroup with Lie algebra $\lieg_0$.
	
	If $(E,\varphi)$ is an $\alpha$-semistable $(G_0,\lieg_1 \oplus \lieg_{-2})$-Higgs pair, with $\varphi = \varphi^+ + \varphi^-$, $\varphi^+ \in H^0(E(\lieg_1) \otimes K_X)$, $\varphi^- \in H^0(E(\lieg_{-2}) \otimes K_X)$, the Toledo invariant $\tau(E,\varphi)$ satisfies the inequalities
	$$-\tau_L \le \tau(E,\varphi) \le \tau_U,$$
	\noindent where
	\begin{align*}
		\tau_L &= \rank\nolimits_T(\varphi^+)(2g-2) - \lambda(2\kappa - \rank\nolimits_T(\varphi^+)),\\
		\tau_U &= \kappa\rank\nolimits_T(\varphi^-)(2g-2) + \lambda(2\kappa - \kappa\rank\nolimits_T(\varphi^-)).
	\end{align*}
	
	Here, $\kappa$ equals $1$ if $\lieg = \liesp_{2n}\C$ and $2$ otherwise.
\end{theorem}
\noindent Proof. The lower bound follows directly from Theorem \ref{amw}, using the bilinear invariant form $B^*$ with $B^*(\beta, \beta) = 2$ and its dual $B$, and noticing that the longest root $\gamma$ with $\lieg_{\gamma} \subseteq \lieg_1$ is always long (i.e. $B(\gamma, \gamma) = 2$) except for the case $\lieg = \liesp_{2n}\C$. For the upper bound, there are two cases. If $\varphi^- = 0$, we can apply Theorem \ref{amw} as well. Otherwise, we can replicate the proof of Theorem \ref{amw} using the prehomogeneous vector space $(G_0,\lieg_{-2})$ instead of $(G_0,\lieg_1)$ and swapping the roles of $\varphi^+$ and $\varphi^-$. For this, following the same notation, the only thing that has to be checked is that $\varphi^+ \in H^0(E_\sigma(\lieg_{1,s}) \otimes K_X)$. First, $(G_0,\lieg_{-2})$ is JM-regular by Lemma \ref{minustwojmreg}. Moreover, since $\varphi^- \neq 0$, it is nonzero generically on $X$. By one-dimensionality of $\lieg_{-2}$, this means that $\varphi^-$ is generically on the open orbit. Thus, $s=0$ and $\lieg_{1,s} = \lieg_1$, so that $\varphi^+ \in H^0(E_\sigma(\lieg_{1,s}) \otimes K_X)$. 

This allows to obtain the conclusion of Theorem \ref{amw}, that is,
\[\tau'(E,\varphi) \ge -\left(\rank_T(\varphi^-)(2g-2) - \lambda(B^*(\beta, \beta)B(\frac{-T_\beta}{2},\frac{-T_\beta}{2}) - \rank_T(\varphi^-))\right),\]
\noindent where $\tau'$ is the Toledo invariant constructed from $(G_0,\lieg_{-2})$, as in Remark \ref{dual}. Using the expression for $\tau'$ in terms of $\tau$, and the remaining known values, we get the desired conclusion. \qed

We will now derive bounds for the Toledo invariant in the moduli space $\mathcal M(G_0, \lieg_1 \oplus \lieg_{-2})$ which are independent of $\varphi$. For this, we set $\lambda = 0$ in Theorem \ref{amwquaternionic} and bound $\rank_T(\varphi^+) \le \rank_T(G_0,\lieg_1)$, $\rank_T(\varphi^-) \le \rank_T(G_0,\lieg_{-2})$. The only remaining task is to compute these Toledo ranks.

As remarked in Lemma \ref{minustwojmreg}, the Vinberg $\C^*$-pair given by $(G_0,\lieg_{-2})$ is JM-regular. Thus, $\rank_T(G_0,\lieg_{-2}) = B^*(\beta, \beta) \cdot B(-\frac{T_\beta}{2},-\frac{T_\beta}{2}) = 1$. Now, the information about the prehomogeneous vector spaces $(G_0,\lieg_1)$ for the quaternionic $\Z$-grading has been tabulated in \cite[Table 2.6]{wallach_quaternionic_1996}. Unless $\lieg = \liesl_n\mathbb C$, they are always given by irreducible representations $\rho : G_0 \to \liegl(\lieg_1)$, and such spaces are classified in \cite[Section 7]{sato_classification_1977}. Within this classification, if $\lieg \neq \liesp_{2n} \mathbb C$, the resulting pair is JM-regular and hence $\rank_T(G_0,\lieg_1) = \kappa \cdot B(T_\beta, T_\beta) = 4$.

If $\lieg = \liesl_n\mathbb C$, the grading corresponds to the case $m=3$ with dimensions $(1,d,1)$ of Example \ref{linear_quiver}. As explained at the end of the referenced example, this case is JM-regular and thus $\rank_T(G_0,\lieg_1) = 4$ as well.

Finally, if $\lieg = \liesp_{2n}\mathbb C$, the resulting space is not JM-regular. We now compute its rank by hand. For this, let $V$ be a $(2d+2)$-dimensional complex vector space with a symplectic form $\omega$. Choose a decomposition $V = V_0 \oplus V_1 \oplus V_2$ with $\dim V_0 = \dim V_2 = 1$, $\dim V_1 = 2d$ and such that $\omega$ restricts to a symplectic form on $V_0 \oplus V_2$ and on $V_1$. The quaternionic $\Z$-grading of $\liesp_{2n}\C$ is given in a similar way to Example \ref{linear_quiver}, by
\[\lieg_k := \left(\bigoplus_{j = 0}^{2}\Hom(V_j, V_{j+k})\right) \cap \liesp_{2d+2}\C.\]
Fix basis $\{e_0\}, \{e_1,\dots,e_d,f_1,\dots,f_d\}, \{f_0\}$ for $V_0$, $V_1$ and $V_2$, respectively, in the usual way with respect to $\omega$ (i.e. $\omega(e_i, f_j) = \delta_{ij}$, $\omega(e_i,e_j) = \omega(f_i,f_j) = 0$). An element of the open orbit $e \in \Omega \subseteq \lieg_1$ is given by $e(e_0) = f_1$, $e(e_1) = f_0$ and $e(v) = 0$ for the remaining basis vectors $v$. This can be completed, with the techniques explained at the end of Example \ref{linear_quiver}, to an $\liesl_2\C$-triple $(h,e,f)$ with $h \in \lieg_0$ defined by $h(f_0) = f_0$, $h(f_1) = f_1$, $h(e_0) = -e_0$, $h(e_1) = -e_1$ and $h(v) = 0$ for the remaining basis vectors $v$. Notice at this point that one possible choice for $T_\beta$ is given by $T_\beta(e_0) = e_0$, $T_\beta(e_1) = -e_1$ and the remaining basis elements sent to zero, which shows that the space is indeed not JM-regular. We then have $\rank_T(G_0,\lieg_1) = \frac{1}{2}\kappa B(T_\beta, h) = \frac{1}{2} \cdot 1 \cdot 2 = 1$.

We have thus proven the following theorem.

\begin{theorem}\label{amwquatcoarse}
	Let $G$ be a complex simple Lie group and $\lieg = \bigoplus_{j=-2}^{2}\lieg_j$ the quaternionic $\Z$-grading on its Lie algebra. Let $G_0 \subseteq G$ be the connected subgroup with Lie algebra $\lieg_0$. Let $(E,\varphi) \in \mathcal M(G_0,\lieg_1 \oplus \lieg_{-2})$ be a $(G_0,\lieg_1 \oplus \lieg_{-2})$-Higgs pair.
	
	The Toledo invariant $\tau(E,\varphi)$ satisfies the inequalities
	\[-4(2g-2) \le \tau(E,\varphi) \le 2(2g-2).\]
	
	Moreover, if $\lieg = \liesp_{2n}\C$ for some $n \ge 2$, then
	\[-(2g-2) \le \tau(E,\varphi) \le (2g-2).\]
\end{theorem}

We end the section by remarking that the associated quaternionic symmetric space $G^\R/H$ in the case $\lieg = \liesp_{2n}\C$ is $\mathbb H\mathbb P^n$, the only one of rank $r=1$ (e.g. \cite[Section 3.1]{bertram_quaternionic_2002}).

\subsection{Cayley correspondence for the quaternionic grading.}

The goal in this section is to show that the Cayley correspondence of Section \ref{cayley} is an isomorphism in the case of cyclic Higgs bundles for the quaternionic grading. That is, for this grading, the locus of cyclic Higgs bundles whose Toledo invariant attains the lower bound in Theorem \ref{amwquatcoarse} (which is the same as in the general result of Theorem \ref{amw}) can be described via the Cayley correspondence. 

As discussed prior to Theorem \ref{amwquatcoarse}, the Vinberg $\C^*$-pair $(G_0,\lieg_1)$ obtained from the quaternionic $\Z$-grading is JM-regular unless $\lieg = \liesp_{2n}\C$. Assume now that $\lieg \neq \liesp_{2n}\C$. JM-regularity implies that the centraliser $C:=(G_0)_e$ for $e \in \Omega \subseteq \lieg_1$ is reductive and one can consider the Cayley correspondence of Section \ref{caycorr}. Moreover, let $V:= \ad^2(e)(\lieg_{-2})$ the vector space of Theorem \ref{cayley}. Since $V \simeq \lieg_{-2}$ as vector spaces, it is one dimensional and we have $[V,V] = 0$. Thus, $(C,V)$ is a Vinberg $\theta$-pair (c.f. Proposition \ref{cayleysurj}) for $\theta$ an involution of order $2$: if $\liec$ is the Lie algebra for $C$, its adjoint representation preserves $V$ and $[V,V] \subseteq \liec$. Then, using Theorem \ref{caycorr} and Proposition \ref{cayleysurj} we have:

\begin{theorem}
	Let $G$ be a complex simple Lie group and $\lieg = \bigoplus_{j=-2}^{2}\lieg_j$ the quaternionic $\Z$-grading on its Lie algebra $\lieg \neq \liesp_{2n}\C$. Let $G_0 \subseteq G$ be the connected subgroup with Lie algebra $\lieg_0$. There is an isomorphism of complex algebraic varieties
	$$\mathcal M^\text{\emph{max}}(G_0, \lieg_1 \oplus \lieg_{-2}) \simeq \mathcal M_{K_X^3}(C,V)$$
	\noindent between the moduli space of polystable $(G_0,\lieg_1 \oplus \lieg_{-2})$-Higgs pairs with Toledo invariant $\tau(E,\varphi) = -4(2g-2)$ and the moduli space of polystable $K_X^3$-twisted $(C,V)$-Higgs pairs.
\end{theorem}

\bibliographystyle{acm}
\bibliography{bibliography}{}
  
\end{document}